\documentclass[a4paper,11pt]{article}
\usepackage[latin1]{inputenc}
\usepackage[T1]{fontenc}
\usepackage[english]{babel}
\usepackage{amsfonts}
\usepackage{amsmath}
\usepackage{enumerate}
\usepackage{amssymb}
\usepackage{stmaryrd}
\usepackage{hyperref}
\usepackage{amscd}
\usepackage{amsthm}
\usepackage[left=3cm,right=3cm,top=4cm,bottom=4cm]{geometry}
\usepackage{fontenc}
\usepackage{esint}
\usepackage{cancel}
\usepackage[pdftex]{graphicx}
\usepackage{xcolor}
\usepackage{mathrsfs}
\newcommand{\ls}{\leqslant}
\newcommand{\gs}{\geqslant}

\newcommand{\R}{\mathbb R}
\newcommand{\Rn}{\mathbb R^n}

\newcommand{\tr}{\operatorname{Tr}}
\newcommand{\per}{\operatorname{Per}}
\newcommand{\Om}{\Omega}
\newcommand{\scal}[2]{\left ( #1 \, , \, #2 \right )}
\newcommand{\scalv}[2]{\left \langle #1 \, , \, #2 \right \rangle}
\renewcommand{\div}{\operatorname{div}}
\newcommand{\reg}{\operatorname{reg}}
\newcommand{\sing}{\operatorname{sing}}

\newcommand{\spt}{\operatorname{spt}}

\renewcommand{\H}{\mathcal H}

\newcommand{\es}{E_s^{(1)}}
\newcommand{\et}{E_t^{(1)}}
\newcommand{\llo}{L^1_{\mathrm{loc}}}
\newcommand{\Cu}{\mathcal C^1}
\newcommand{\dist}{\operatorname{dist}}
\newcommand{\dpa}[2]{\frac{\partial #1}{\partial x_{#2}}}
\newcommand{\dpp}[3]{\frac{\partial^2 #1}{\partial x_{#2} \partial x_{#3}}}
\newcommand{\dpc}[2]{\frac{\partial^2 #1}{\partial x_{#2}^2}}

\newcommand{\eps}{\varepsilon}
\newtheorem{thm}{Theorem}
\newtheorem{defi}{Definition}
\DeclareMathOperator*{\argmin}{arg\,min}
\DeclareMathOperator*{\esup}{ess\,sup}
\newtheorem{prop}{Proposition}
\newtheorem{cor}{Corollary}
\newtheorem{lem}{Lemma}

\theoremstyle{remark}
\newtheorem*{rem}{Remark}
\title{Continuity results for $TV$-minimizers}
\author{Gwenael Mercier\footnote{Johann Radon Institute for Computational and Applied Mathematics (RICAM), Austrian Academy of Sciences, Altenbergerstra\ss{}e 69, A-4040 Linz, Austria. \texttt{gwenael.mercier@ricam.oeaw.ac.at}}}
\date{}

\begin{document}
\maketitle
\begin{abstract}This paper deals with continuity preservation when minimizing generalized total variation with a $L^2$ fidelity term or a Dirichlet boundary condition. We extend several recent results for these two types of data terms, mainly by showing comparison principles for the prescribed mean curvature problem satisfied by the level-sets of such minimizers. \end{abstract}


 \section{Introduction}
In this paper, we study the regularity of minimizers of generalized total variations. More precisely, let $\Omega$ be a subset of $\Rn$ and $g$ be a function which is defined on some subset $\tilde \Omega$ of $\overline \Omega.$ We want to analyze the regularity of minimizers of
\begin{equation} \min_{u \in BV} \int_\Omega F(\nabla u ) \label{genrof}\end{equation}
where $F : \Rn \to \R$ is a convex function with linear growth ($\frac{1}{\mu} |x| \ls F(x) \ls \mu |x|$) and with two possible links between $u$ and $g$:
\begin{enumerate}
 \item Either a Dirichlet condition $u=g$ on $\partial \Omega$, which can be motivated by Current Density Imaging \cite{NacTamTim11} (even if that would require a space dependent anisotropy) or Mechanics \cite{Ray93},
\item Or a $L^2$-distance between $u$ and $g$
$$d(u,g) = \int_\Omega \frac{(u-g)^2}{2}, $$ which is the distance introduced by Rudin, Osher and Fatemi in \cite{rudin92} in their well known denoising model. 

\end{enumerate}
Regularity results for minimizers of \eqref{genrof} constitute a wide literature. The pioneer work of Miranda \cite{miranda65} has been generalized by Bousquet and Clarke (see \cite{clarke05,BouCla07,bousquet07,bousquet15}) whereas Bildhauer \cite{Bil00} (and previously Seregin \cite{Ser90}) study this minimization problem taking advantage of its dual formulation (see also \cite{Ser96} for a physical interpretation of the dual variable and \cite{beck13} for a link between these two approaches).

Using the theory of $BV$ functions and finite perimeter sets, Sternberg, Williams and Ziemer use in \cite{SteWilZie92,SteWilZie93} the coarea formula to link the minimizing property of $u$ in \eqref{genrof}, for $F$ the Eucliden norm, and a geometrical problem of minimal surface type satisfied by its level-sets $\{u > t\}.$ This allows them to use all the techniques from Geometric Measure Theory, and more precisely minimal currents, in particular the then-recently proven comparison principle between minimal surfaces \cite{simon87} (see also \cite{moschen77} and \cite{solomon89,Ilm96} when surfaces are only critical points for area) to show a first continuity result for total variation minimization. In this paper, we make an intensive use of this geometrical problem (but in an anisotropic framework) in the spirit of the papers by Caselles, Chambolle, Novaga \cite{chambolle07,chambolle11} (see also \cite{jalalzai12,Jal13,Jal14,Jal16}) for the $L^2$ distance, and \cite{jerrard13,NacTamTim11} for the Dirichlet condition (see also \cite{MorNacTam14}). See also the references therein for more work on this subject.

In the following, we try to give a rather self contained presentation and we will recall some results that are already part of literature. Nonetheless, we formulate them in the simplest form that fits our needs. It has to be noticed that we will use a little of anisotropic geometry, that has been introduced and developed in recent years, see \cite{AmaBel94,Bel04} for details.

\vspace{0.5cm}
Let us now give a brief overview of this article. In what follows, we will be interested in three types of $F$:
\begin{enumerate}
\item $F(\nabla u ) = |\nabla u|,$ that is the usual total variation,
\item $F(\nabla u) = \varphi(\nabla u)$, where $\varphi$ is a norm in $\Rn$, which can be non Euclidean,
\item $ F(\nabla u) = f(\varphi(\nabla u))$, where $f : \R^+ \to \R^+$ is a convex function and $\varphi$ is a norm in $\Rn$.
\end{enumerate}
The current framework will be recalled in every section.

All along this paper, our goal is to link the regularity of $u$ with the regularity of $g$. More precisely, we want to show that the minimizing procedure \eqref{genrof} preserves continuity. One can even show (see Section \ref{secuc}) that under strong assumptions on the domain, we can control the modulus of continuity of $u$ by the modulus of $g$. 

Let us present now the structure of this article.
\begin{itemize}
 \item In a first section, we present briefly the work by Miranda \cite{miranda65}, which shows typical behavior of minimizers of \eqref{genrof} and introduce $BV$-functions and sets with finite perimeter. In particular, we define the quantity $F(\nabla u)$ for $BV$ functions and we formulate the link between minimizers of \eqref{genrof} and geometric minimizers of
\begin{equation} \per_\varphi(E;\Omega) - \int_{E\cap \Omega} g \label{formgeom}, \end{equation}
which is the variational formulation of ``$E$ has a prescribed mean curvature $g$.''
We also give some density properties of these geometrical minimizers.
We finally recall the known regularity results on $u$ which deal with its jump set (hypersurfaces of discontinuity).
\item In Section \ref{secuc}, we apply Miranda's scheme to study minimizers of 
$$\int_\Omega f(\varphi(\nabla u)) + \frac{(u-g)^2}{2}$$ with Neumann boundary conditions in a convex domain, and we show that the control of the modulus of continuity of $u$ on the boundary can be obtained using these boundary conditions. We can then obtain a bound on the modulus on the whole domain, which constitutes an extension of a result by Caselles, Chambolle and Novaga \cite{chambolle11} to higher dimension and anisotropic framework.
\end{itemize}

In the sections which follow, we use level-sets $E_s = \{u>s\}$ and their minimizing property to get regularity results for $u$. Indeed, showing that $u$ is continuous is equivalent to show that $\partial E_s \cap \partial E_t = \emptyset$ as soon as $s \neq t$.

\begin{itemize}
 \item We first show in Section \ref{onesmooth} that one can quite easily extend the usual Hopf maximum principle for $\mathcal C^2$ geometric minimizers of \eqref{formgeom} to the case where only one of the two minimizers is regular. This result is known for $g=0$ \cite{solomon89} but we give a much simpler proof in the spirit of Caffarelli, Cordoba, Roquejoffre, Savin \cite{caffarelli93,caffarelli10}.
\item In Section \ref{secjer}, we investigate the problem $$\min \int_\Omega \varphi(\nabla u)$$ in bounded domains with continuous Dirichlet boundary conditions. We could use the scheme of Miranda, but since the functional is no longer strictly convex, we have to find another way to get a comparison principle for minimizers ($u \ls v$ on $\partial \Omega$ implies $u \ls v$ in the whole $\Omega$): Jerrard, Moradifam and Nachman proposed a geometric proof of this principle in \cite{jerrard13} (inspired by \cite{SteWilZie92}), with an strict $\varphi$-mean convexity assumption on the domain $\Omega$. Since \cite{jerrard13} deals with a space dependent $\varphi$, they can obtain continuity of the minimizer only in dimension $\ls 3,$ when the level-sets of the minimizer are regular. Taking advantage of the translation invariance, we prove continuity for $u$ in all dimensions, using simpler arguments than in \cite{jerrard13}. Nonetheless, the proof is totally geometric (it deals with level-sets) and remains in the spirit of \cite{jerrard13} and \cite{SteWilZie92}.
\item Finally, in Section \ref{seclocal}, we come back to the usual Rudin-Osher-Fatemi model (no anisotropy). We show that some results can be stated in a generic open set but that the situation is more difficult, because we cannot use the boundary as a step towards continuity. As a result, we show that a strong maximum principle for minimal surfaces \cite{simon87} can be extended to variational constant mean curvature hypersurfaces, and see that it is enough to claim that two different level-sets of a minimizer cannot touch. That is exactly proving that the minimizer $u$ is continuous.
\end{itemize}
 \section{Tools and related results}
 \label{defiso}
\subsection{The pioneer work by miranda}
We present briefly here one of the first papers on minimizing
\begin{equation} \int_\Omega F(\nabla u), \quad u=g \text{ on } \partial \Omega, \label{eqmir}\end{equation}
which was published in Italian by Miranda \cite{miranda65}. Miranda assumes here that $F$ is $\mathcal C^2$ and strictly convex and that the domain $\Omega$ is open and bounded. In addition, the boundary data $g$ satisfies the so called $K$-bounded slope condition (BSC): for every $ \hat y \in \partial \Omega$, there exist two affine functions $f^\pm$, vanishing at $\hat y$, such that for every other $y' \in \partial \Omega$, we have
\begin{equation} f^-(y') \ls g(y')-g(\hat y) \ls f^+ (y').\label{eqbsc} \end{equation}
The main statement of \cite{miranda65} is
\begin{thm} \label{thmmiranda65}
 There exists a unique minimizer of \eqref{eqjer} in the class of Lipschitz functions.
\end{thm}
There is no work on $BV$ (or even in $W^{1,1}$) functions in this article: every function is at least continuous. Nonetheless, the techniques used to show this theorem are fundamental in this whole paper. Let us give a few words about the proof.

First, since $F$ is strictly convex, there is at most one (Lipschitz) minimizer of \eqref{eqjer}. And we have the following comparison principle, which directly follows from the strict convexity of $F$.
\begin{prop}
 Let $u$ and $v$ two minimizers of \eqref{eqjer} with boundary data $g$ and $h$. Then, if $g \ls h$, $u \ls v.$
\label{jercomp}
\end{prop}
To show the existence, Miranda minimize \eqref{eqmir} in the classe of $p$-Lipschitz functions, providing some function $u_p$. To make the sequence $(u_p)$ converge, he shows that they actually all share a Lipschitz constant. This is a regularity result which will be fundamental in what follows.

Thanks to the (BSC) and Proposition \ref{jercomp}, one controls the behavior of a minimizer on the boundary. Indeed, since $f^\pm$ are affine, they are natural minimizers of $\int F(\nabla u).$ The proposition above applied to $u$ and $f^\pm$ shows, after straightforward computations,
\begin{equation}  \forall (y,\hat y) \in \overline \Omega \times \partial \Omega, \quad |u(y) - u(\hat y)| \ls K |y-\hat y|. \label{controlb} \end{equation}

The most important result is that the control of the reguarity of a minimizer comes directly from the control on its boundary: let $u$ be a minimizer of \eqref{eqjer} which satisfies \eqref{controlb}. Then, it is $K$-Lipschitz.
To prove it, Miranda uses that the translational invariance of the integral. If $y'$ and $y$ are two points in $\Omega$, one defines $v(x):=u(x+(y'-y)).$ Thanks to the comparison principle, $\max (u-v)$ is reached on the boundary of $\overline{\Omega \cap (\Omega + y'-y)},$ in some $\hat x$. That yields the expected result after simple computations.

Finally, let us make a remark on the bounded slope condition:
\begin{rem}
 Let us assume that $\Omega$ is uniformly convex and $g$ is $\mathcal C^2$. Then, $g$ satisfies the BSC. Pierre Bousquet proved in \cite{bousquet07} that if $g$ is only continuous, Theorem \ref{thmmiranda65} still holds (in the class of continuous functions instead of Lipschitz ones). The idea is to approximate $g$ by $\mathcal C^2$ functions $g_i$ and control the Lipschitz norms of the approximate minimizers. In addition, Bousquet deals with functions in $W^{1,1}.$ See also \cite{bousquet15} for a generalization where affine functions are no longer minimizers.
\end{rem}
In the sequel, we work with non strictly convex functionals and therefore we work in the class of functions with bounded variation.
\subsection{Functions with bounded variation}
\begin{defi}
Let $u \in L^1(\Omega ; \mathbb R)$. We say that $u$ has bounded variation and note $u \in BV(\Omega)$ if its distributionnal derivative $Du$ is a Radon measure. Then, we call $TV(u, \Omega)$ the norm of this derivative, as a Radon measure:
$$TV(u; \Omega) = \sup_{\phi \in \mathcal C^1_c(\Omega), \; \Vert \phi \Vert_{L^\infty} \ls 1 } \scalv{Du}{\phi}=\sup \left \{ \int_\Omega u \div \phi \, \middle | \, \phi \in \mathcal C^\infty_c(\Omega, \Rn),\; \Vert \phi \Vert_{L^\infty} \ls 1. \right \}.$$
\end{defi}

\begin{defi} \label{caccio}
Let $E$ be a mesurable set in $\Rn$. We say that it has finite perimeter in $\Omega$ if its characteristic function $1_E$ has bounded variation in $\Omega$. We note
$$\per(E;\Omega) := TV(1_{E};\Omega).$$
If $\Omega = \Rn$, we write $\per(E).$
\end{defi}
\begin{prop}
 Let $A$ and $B$ two finite perimeter sets. Then,
$$\per(A\cap B;\Omega) + \per (A\cup B ;\Omega) \ls \per(A;\Omega)+\per (B; \Omega).$$
\label{unint}
\end{prop}
For every finite perimeter set $E \subset \Omega$, we note $E^{(1)}$ the set of points with density $1$ and $E^{(0)}$ the set of points with density zero. More precisely,
$$E^{(1)} := \left \{ x \in \Omega\, \middle | \, \lim\limits_{r \to 0}  \frac{\left \vert B_r(x) \cap E \right \vert}{\left \vert B_r \right \vert} = 1 \right \},$$
$$E^{(0)} := \left \{ x \in \Omega\, \middle | \, \lim\limits_{r \to 0}  \frac{\left \vert B_r(x) \cap E \right \vert}{\left \vert B_r \right \vert} = 0 \right \}.$$
 These sets are invariant to negligible modifications of $E$ and we have 
 $$|E \Delta E^{(1)}| = |E^c \Delta E^{(0)}| = 0.$$

\begin{defi}[Reduced boundary]
A point $x \in \Omega$ belongs to the reduced boundary of $E$ (we note $x\in \partial^\ast E$) if
\begin{enumerate}[i)]
\item For every  $\rho >0$, $\int_{B_\rho(x)} |D1_E| >0$.
\item The quantity
$$\nu_\rho (x) = \frac{\int_{B_\rho(x)} D1_E}{\int_{B_\rho(x)} |D1_E|}$$
has a limit $\nu(x)$ with $|\nu(x)| = 1$ when $\rho \to 0$.
\end{enumerate}
Then, we have 
\begin{equation}
 \label{integred}
\per(E;\Omega) = \int_{\partial^\ast E \cap \Omega} d \mathcal H^{n-1}(x).
\end{equation}
\end{defi}

\subsection{Rudin-Osher-Fatemi denoising procedure}
In 1992, Rudin, Osher and Fatemi proposed in \cite{rudin92} a denoising procedure based on total variation minimization. More precisely, if $g : \Omega \subset \Rn \to \R$ is a noisy picture, they suggest to regularize it solving
\begin{equation} u = \argmin_{v \in BV(\Omega)} \int_{\Omega} |D u| + \frac{1}{\lambda} \int_\Omega \frac{(u-g)^2}{2} \label{rof}\end{equation}
where $\lambda$ is a real parameter.

In what follows, we are interested in anisotropic generalizations of this problem. More precisely, let $\varphi$ be a smooth, symmetric $(\varphi(-x) = \varphi(x))$ anisotropy (a norm in $\Rn$) such that $\varphi^2$ is strongly convex, we deal with
\begin{equation}
u = \argmin_{v \in BV(\Omega)} \int_{\Omega} \varphi(D u) + \frac{1}{\lambda} \int_\Omega \frac{(u-g)^2}{2}.
\label{arof}
\end{equation}
In this equation, the term 
$$ \int_{\Omega} \varphi(D  u) $$ has to be understood as 
$$\int_{\Omega} \varphi\left(\frac{D u}{|D u|}\right) \operatorname{d}(|D u|)(x) = \sup \left \{ \int_\Omega u \cdot \div \xi \; \middle \vert \; \varphi^\circ(\xi) \ls 1 \right \}$$
where $D u$ is the derivative of the $BV$-function $u$, and $\frac{Du}{|Du|}$ is the Radon-Nikodym derivative of $Du$ with respect to $|D u|.$ $\varphi^\circ$ is the polar of $\varphi$ and is defined as
$$ \varphi^\circ (\xi) = \sup\{ \scalv{x}{\xi} \ \vert \ \varphi(x) \ls 1 \}.$$
Since the functional $u \mapsto \int_{\Omega} \varphi(D u) + \frac{1}{\lambda} \int_\Omega \frac{(u-g)^2}{2}$ is strictly convex and semi continuous (thanks to the semi continuity of the total variation), it has a unique minimizer in $BV(\Omega).$

In all the following, we are searching for the links which may exist between the regularity of $g$ and the regularity of $u$.

\subsection{On the level-sets of minimizers}
In this subsection, we give a few results which link the minimizing property of $u$ in \eqref{arof} and the minimizing property of each level-set of $u$
\begin{equation}
 E_t := \{u > t\}.
\label{levelset}
\end{equation}
To this aim, let us introduce some anisotropic variants of the quantities presented above. 
\begin{defi}[Anisotropic perimeter]
 Let $E$ has finite perimeter. We can define the anisotropic $\varphi$-perimeter by
\begin{equation}
\per_\varphi(E;\Omega) := \int_{\partial^\ast E \cap \Omega} \varphi(\nu_E) d \mathcal H^{n-1}.
 \label{anisper}
\end{equation}
\end{defi}
Note that if $\varphi = Id$, then, thanks to \eqref{integred}, we obtain the usual perimeter.
It is easy to show that $\per_\varphi$ satisfies the same properties as the isotropic perimeter (with the same proofs which basically use the semi continuity of the total variation with respect to the $L^1$ convergence). For instance,
\begin{equation}\per_\varphi( E \cap F) + \per_\varphi(E\cup F) \ls \per_\varphi(E) + \per_\varphi(F) \label{addperv}\end{equation}
and the key-tool in what follows, the so called anisotropic coarea formula (see \cite[Th. 1.23]{giusti84} for the isotropic case and \cite[Rem. 4.4]{AmaBel94} for the anisotropic one) which leads to the
\begin{prop}
\label{proparof}
 Let $u \in BV(\Omega).$ Then, $u$ minimizes \eqref{arof} with $\lambda = 1$ if and only if for every $t \in \R$, the level sets $E_t$ of $u$ minimize
\begin{equation} F \mapsto \per_\varphi(F) + \int_F t-g. \label{parof}\end{equation}
\end{prop}

\subsection{Jump-set}
Let us state here the first regularity results on $u$ which come from regularity of $g$. They link the jump set $J_u$ of the solution to the one of the data $J_g$. These results come from \cite{chambolle07} for the isotropic version and \cite{jalalzai12} for the anisotropic (with space dependency) generalization.
\begin{thm}[Caselles, Chambolle, Novaga, '07 and Jalalzai, '12]
 Let $g \in BV(\Omega) \cap L^\infty(\Omega)$ where $\Omega \subset \Rn$, and let $u$ minimize \eqref{arof}. Then
$$ J_u \subset J_g$$
up to a $\H^{n-1}$-negligible set.
\end{thm}

Finally, we mention a recent paper by Valkonen \cite{valkonen14}, which extends this results to much more general regularizations.

\subsection{Density estimates on the minimizers of the geometric problem}
In this subsection, we give useful results on the minimizers of the anisotropic perimeter. The main density estimate comes from \cite{gonzales93}, with slight changes due to the anisotropic framework. It can be noticed that these estimates are also valid in the non local framework (see \cite{caffarelli10,CafVal13}).
\paragraph{A word on the anisotropies.} In this subsection, we will use an anisotropy $\varphi$. It just consists in a norm in $\Rn$. We assume that it is smooth and that $\varphi^2$ is strongly convex ($D^2 \varphi^2 \gs \lambda I$ with $\lambda >0$). As a result, there exist two constants $A$ and $B$ such that 
$$\forall |x| = 1, \quad A|x| \ls \varphi(x) \ls B|x|.$$
This inequality allows most of standard isotropic estimates to remain in this anisotropic framework.

\begin{prop}
 Let $E$ minimize \eqref{parof} in $B_1$ and assume that $0 \in \partial E$. Then, there exist $r_0$ and a constant $q >0$ which both depend on the dimension, $\Vert g-t\Vert_\infty$ and $A$ and $B$ such that for every $r \ls r_0$,
\begin{equation}\label{densest} 1-q \gs \frac{|B_r \cap E|}{|B_r|} \gs q \end{equation}
\label{densite1}
\end{prop}
Let us state a corollary which will be useful in what follows. This corollary is often mentionned as \emph{clean-ball property} (see \cite{caffarelli10} and related work).

 \begin{cor}
Let $E$ be minimizing in $B_1$ with $0 \in \partial E$. Then, there exists $q >0$ (depending only on $A$, $B$, $\Vert t-g \Vert_\infty$ and the dimension) such that for all $r \ls r_0$ there exists a ball $B_{qr} \subset E\cap B_r$ of radius $qr.$ In addition, there exists another ball $B'_{qr}$ with the same radius, such that $B'_{qr} \subset \Rn \setminus E \cap B_r.$
\label{cleanball}
\end{cor}
Finally, these density estimates give some information on the points of density one.
\begin{prop}
\label{propdensone}
 Let $E$ be a minimizer of \eqref{parof}. Then, the sets $E^{(1)}$ of points with density 1 in $E$ and $E^{(0)}$ of points with density $0$ in $E$ are both open subsets of $\Rn$.
\end{prop}

This observation enables to define
\begin{defi}
\label{defupm}
 Let $u$ be a minimizer of \eqref{arof} and let $E_s := \{u > s\}$ its level-sets. Then, we can define two particular representatives for $u$, denoted by $u^+$ and $u^-$, such that
 $$ \{ u^+ > s\} := E_s^{(1)} \qquad \text{ and } \qquad \{u^- \gs s\} := \left( E_s^{(0)}\right)^c.$$
 Then, we have $u^\pm = u$ a.e., $u^- \ls u \ls u^+$, $u^+$ is lower semicontinuous whereas $u^-$ is upper semicontinuous.
\end{defi}

Now, we are ready to give the main regularity results of this article. Let us begin by a theorem really in the spirit of Miranda's work.
\section{On convex domains with Neumann boundary conditions}
\label{secuc}
In this section, we apply Miranda's scheme to study \eqref{arof} with Neumann boundary conditions. The assumption of convexity of $\Omega$ allows to obtain the control of the modulus of continuity on the boundary, as we will see. 

\begin{thm}
Let $\Omega$ be a convex bounded domain and $f : \R \to \R$ be convex and satisfies $f(0)=0$ and $f(+\infty) = + \infty$.
Let $u$ be the minimizer, among all functions with bounded variations in $\Omega$, of 
$$\int_\Omega f(\varphi(\nabla u)) + \frac{(u-g)^2}{2}$$
with free boundary condition and assume that $g$ is continuous with $\varphi$-modulus $\omega$, that is
$$ \forall x,y \in \Omega, \quad |g(x) - g(y)| \ls \omega( \varphi^\circ (x-y))$$
where
$$\varphi^\circ(\xi) = \sup\left \{ \scalv{x}{\xi} \ \middle \vert \ \varphi(x) \ls 1 \right \}.$$
Then, $u$ is continuous with $\varphi$-modulus $\omega$.
\label{thmneu}
\end{thm}
This theorem extends \cite[Th. 5.1]{chambolle11} to anisotropic framework and higher dimension (indeed, \cite{chambolle11} makes use of the regularity of the level-sets of $u$, which occurs only in low dimension: see \cite{SchSimAlm77}).

\begin{rem}
Note that it is enough to show Th. \ref{thmneu} for $\Omega$ strictly convex, $\varphi$ smooth uniformly elliptic (that is $\varphi^2 $ is strongly convex) and $g$ smooth. We can indeed approximate any norm $\varphi$ by $\varphi_n \ls  \varphi$ uniformly on compact subsets with $\varphi_n$ satisfying these properties. Noting that if $g$ is $\varphi$-continuous with modulus $\omega$, it is then $\varphi_n$-continuous with the same modulus (since $\varphi_n^\circ \gs \varphi^\circ$) and Th. \ref{thmneu} extends to any norm (even crystalline). \\
Moreover, it is easy to show that if $\Omega_n \to \Omega$ in Hausdorff distance, then the corresponding minimizers $u_n$ of \eqref{arof} in $\Omega_n$ converge to the minimizer $u$ of \eqref{arof} in $\Omega.$ \\
Finally, by approximation as well, we can assume that $g$ is smooth.
\end{rem}

Before proving this theorem, let us make a remark which somehow links the $\varphi$ modulus and the standard one (we are more likely to know the latter).
\begin{rem}
 \begin{itemize}
  \item If we work in Euclidean geometry ($\varphi = Id$), then the $\varphi$-modulus is nothing but the usual modulus.
  \item Since $\varphi^\circ$ is a norm, it is equivalent to the Euclidean one, so there exists $\mu >1$ such that 
$$\forall x \in \Omega, \quad \frac 1\mu |x| \ls \varphi^\circ(x) \ls \mu |x|.$$ 
Thus, if $g$ is continuous with usual modulus $\omega$, one can introduce 
$$\tilde \omega(r) := \omega(\mu r),$$
which satisfies
$$\forall x,y \in \Omega, \quad \tilde \omega ( \varphi^\circ (x-y)) = \omega (\mu \varphi^\circ (x-y)) \gs \omega( |x-y|)$$
and apply Theorem \ref{thmneu} with $\tilde \omega.$
 \end{itemize}
\end{rem}

The strategy of the proof is to work on the approximate problem 
\begin{equation} \min_u \int_{\Omega} f_\eps(\varphi(\nabla u)) + \frac{(u-g)^2}{2} \label{eq:apmin} \end{equation}
with $f_\eps \to f$, locally uniformly and $f_\eps \gs 0$, smooth and satisfies
$$ \varepsilon \ls f_\eps'' \ls \frac{1}{\eps}$$
as well as $f'_\eps(0) = 0.$
\begin{lem}
\label{lem:EL}
 The approximate minimizer $u_\eps$ of \eqref{eq:apmin} is continuous on $\overline \Omega$ and satisfies the Euler-Lagrange equation on $\partial \Omega$
 \begin{equation}  f_\eps'(\varphi(\nabla u))  (\nabla \varphi(\nabla u) \cdot \nu) = 0 \label{eqanibound} \end{equation}
 in the viscosity sense ($\nu$ being the outer normal to $\Omega$).
\end{lem}
We recall that the \emph{viscosity sense} means that if $\psi$ is a smooth function and $u-\psi$ reaches a maximum (resp. a minimum) at $\hat x \in \partial \Omega$, then
$$ f_\eps'(\varphi(\nabla \psi (\hat x)))  (\nabla \varphi(\nabla \psi (\hat x)) \cdot \nu(\hat x)) \ls 0 \quad \text{ (resp.} \gs 0).$$
See \cite{c92user} for an introduction to these notions, in particular Section 7 for generalized boundary conditions.
\begin{rem}
 In the isotropic case, $u_\eps$ satisfies an elliptic equation and therefore, $u_\eps$ is $\mathcal C^\infty$ up to the boundary and Equation \eqref{eqanibound} is satisfied classically and reads
 $$\nabla u_\eps \cdot \nu = 0 \quad \text{on } \partial \Omega.$$
\end{rem}

Note that the proof of \eqref{eqanibound} reveals the typical link between minimizing functionals and viscosity solutions. It will appear again in Section \ref{onesmooth} and has been extensively used in this type of problems \cite[Lemma 2]{caffarelli93} and more recently \cite{caffarelli10,chambolle122,thouroude12,mercier}, the three last references dealing with time dependent equations.
\begin{proof}
We first prove that $u_\eps$ is continuous. We will use \cite{giaquinta82} (the direct application of Theorem 3.1 gives interior continuity). We first note that Equality (3.2) of \cite{giaquinta82} holds for $u_\eps$ and even if the balls $B_R$ and $B_\rho$ cross the boundary,
\begin{equation}\label{eq32gg}
\forall \rho \ls R,\, \forall k,\quad \int_{\{u_\eps>k\} \cap B_\rho} |Du_\eps|^2 \ls \frac{C}{(r-\rho)^2} \int_{\{u_\eps>k\} \cap B_R} (u_\eps-k)^2.
\end{equation}

We want to obtain the continuity up to the boundary by applying \cite[Th. 6.1]{uraltseva68}. Nonetheless, this theorem only provides interior regularity.

To be able to obtain boundary regularity for $u_\eps$, we extend it in the following way. The boundary $\partial \Omega$ is smooth, so for every $\hat x \in \partial \Omega$, there exists a ball $B_r( \hat x)$ and a function $g$ such that 
$$ \Omega \cap B_r(\hat x)  = \{ (x',x_n) \ \vert \ x_n \ls g(x')\}.$$
As a result, by stating (with $x = (x',x_n)$)
$$\tilde u_\eps(x',x_n) = \left \{ \begin{matrix} u_\eps(x) \text{ if $x \in \Omega \cap B_r(\hat x) $ } \\ u_\eps(x',2g(x') - x_n) \text{ if $x \in  \Omega^c \cap B_r(\hat x)$ } \end{matrix} \right.,$$
we locally extend $u_\eps$ in the whole ball $B_r$. 
Thus equality \eqref{eq32gg} (for balls included in $B_r(\hat x)$) also holds for $\tilde u_\eps$ (with a different but controlled $C$). We conclude, as in \cite{giaquinta82}, using \cite[Th. 6.1]{uraltseva68} which provides the continuity of $\tilde u_\eps$ on the interior of $B_r(\hat x)$. In particular, covering all $\partial \Omega$ with such balls, we conclude that $u_\eps$ is continuous on $\overline \Omega.$

Now, let us prove \eqref{eqanibound}.

We denote $F(p) = f(\varphi(p))$, which is
$C^\infty$ away from $0$ and satisfies: $D^2F \ge \gamma I$,
 $D^2 F$ is bounded.

Assume that there exists a smooth function $\psi$ such that $u_\eps\le\psi$, and $u_\eps(\bar x)=\psi(\bar x)$ for some
$\bar x\in\partial\Om$ (we assume the contact is unique).
We assume in addition that
$\nabla F(\nabla\psi(\bar x))\cdot \nu_{\bar x}>0$
and will try to reach a contradiction. Note first that since $\nabla F (0) = 0$, we have $\nabla \psi(\bar x)) \neq 0.$

For $x\in\Om$, we denote $d(x) = \textup{dist}(x,\partial\Om)$.
For $\delta>0$, $\lambda>0$ fixed, we let
\[
\psi_\delta = \psi - \frac{\lambda }{2} \frac{[(\delta-d)^+]^2}{\delta}
\]
which is $\psi$, hence larger than $u_\eps$, when $d\ge \delta$,
while it is $\psi-\frac{\lambda \delta}{2}$ on $\partial\Om$ so that
$\psi_\delta(\bar x)<u_\eps(\bar x)$. We denote
$w_\delta= (u_\eps-\psi_\delta)^+$ and $A_\delta=\{w_\delta>0\}\subset
\{d\le\delta\}$.

We remark that if $x\in A_\delta$ so that $d\le \delta$,
\[
\nabla \psi_\delta = \nabla \psi + \lambda\left(1-\frac{d}{\delta}\right)^+\nabla d
\]
and in particular, on $\partial \Om$ (using $d=0$, $\nabla d= -\nu$) one
has
\[
\nabla\psi_\delta = \nabla\psi-\lambda\nu,
\]
while if $d=\delta$, $\nabla\psi_\delta=\nabla\psi$.
Then (still for $d\le \delta$), 
\[
D^2\psi_\delta = D^2\psi 
+ \lambda\left(1-\frac{d}{\delta}\right)^+D^2 d
- \frac{\lambda}{\delta}\nabla d\otimes\nabla d.
\]
Observe in particular that in the same set,
\begin{multline}\label{eq:coercivitydelta}
-\div\nabla F(\nabla\psi_\delta)
= -D^2F(\nabla\psi_\delta):D^2\psi_\delta
\\
= -D^2F(\nabla\psi_\delta):\left[D^2\psi 
+ \lambda\left(1-\frac{d}{\delta}\right)^+D^2 d \right]
+ \frac{\lambda}{\delta}(D^2F(\nabla\psi_\delta)\nabla d,\nabla d)
\\
\ge -C +\frac{\lambda\gamma}{\delta} \ge 2\|g\|_\infty
\end{multline}
if $\delta$ is small enough (here $C$ is a bound for many quantities,
such as the curvature of $\partial\Om$ near $\bar x$).

The minimality of $u_\eps$ ensures that
\[
\int_\Om F(\nabla u_\eps) + \frac{(u_\eps-g)^2}{2}dx
\le
\int_\Om F(\nabla (\psi_\delta\wedge u_\eps)) + \frac{(\psi_\delta \wedge u_\eps-g)^2}{2}dx
\]
or equivalently
\[
\int_{A_\delta}F(\nabla u_\eps)+\frac{(u_\eps-g)^2}{2}dx
\le
\int_{A_\delta}F(\nabla \psi_\delta)+\frac{(\psi_\delta-g)^2}{2}dx.
\]
Using the (strong) convexity of $F$ and $t\mapsto (t-g)^2/2$, it
follows
\[
\int_{A_\delta}\nabla F(\nabla \psi_\delta)\nabla w_\delta + (\psi_\delta-g) w_\delta
dx
+ \frac{1}{2}\int_{A_\delta}\gamma |\nabla w_\delta|^2+w_\delta^2 dx
\le 0
\]
which, integrating by parts (and using $w_\delta=0$ on $\partial A_\delta
\cap\Om$) yields
\[
\int_{A_\delta\cap \partial\Om} w_\delta\nabla F(\nabla \psi-\lambda\nu)\cdot\nu
+
\int_{A_\delta} \left(-\div\nabla F(\nabla \psi_\delta) + (\psi_\delta-g)
\right)w_\delta dx \le 0.
\]
Now, using~\eqref{eq:coercivitydelta}, we observe that in $A_\delta=\{w_\delta>0\} =\{\psi_\delta < u_\eps\}$, for $\delta$ small enough, we have
\[
-\div\nabla F(\nabla \psi_\delta) + (\psi_\delta-g) \ge 0
\]
so that we obtain
\[
\int_{\partial\Om} w_\delta\nabla F(\nabla \psi-\lambda\nu)\cdot\nu \le 0.
\]
If 
$\nabla F(\nabla\psi(\bar x))\cdot \nu_{\bar x}>0$, then we can choose
$\lambda$ small such that  in a neighborhood of $\bar x$,
$\nabla F(\nabla \psi-\lambda\nu)\cdot\nu>0$ (we use the fact
that $\nabla\psi(\bar x)\neq 0$ and that $\nabla F$ is continuous
away from $0$). We clearly obtain a contradiction if
$\delta$ is small enough, as it should follow  (observing that
$A_\delta$ has nonempty interior, as $\psi_\delta(\bar x)<u_\eps(\bar x)$,
and goes to $\{\bar x\}$ as $\delta\to 0$)
that
\[
\int_{\partial\Om} w_\delta\nabla F(\nabla \psi-\lambda\nu)\cdot\nu > 0.
\]

\end{proof}

Finally, let us conclude the proof of Theorem \ref{thmneu}. We assume that the modulus of continuity $\omega$ is smooth away from zero, with $\omega' > 0$. We introduce
$$L = \sup_{x,y \in \overline \Omega} \frac{u_\eps (x) - u_\eps(y)}{\eps + \omega(\varphi_\eps^\circ(x-y))}.$$
Since $u_\eps$ is continuous on $\overline \Omega$, it is reached at $\hat x \neq \hat y$. We will need the following

\begin{lem}
Either $L \ls 1$ or $L$ is reached on the boundary of $\Omega$.
\label{lemneub}
\end{lem}

\begin{proof}
First, note that this supremum is a maximum, since $\frac{u_\eps(x) - u_\eps(y)}{\eps+ \omega(\varphi_\eps^\circ(x-y))} \to 0$ as soon as $|x-y| \to 0.$ 

Let us now assume (to get a contradiction), that $L > 1$ and that the maximum is not reached on the boundary. That is, we assume that there exists $\delta >0$ such that 
$$\sup_{\substack{x \in \partial \Omega \\ y\neq x \in \overline \Omega}} \frac{|u_\eps(x) - u_\eps(y)|}{\eps + \omega(\varphi_\eps^\circ(x-y))} \ls L- \delta.$$ 
We can choose $\delta$ such that $L -  \delta >1$. 
Letting $$v_\eps = u_\eps(\cdot -z) - (L-\delta) (\eps+\omega(\varphi_\eps^\circ(z))),$$ we have just said that $v_\eps \ls u_\eps$ on $\partial \Omega_z$ where $\Omega_z = (\Omega + z) \cap \Omega.$ Using the very definition of $u_\eps$, one can write (on $\Omega \setminus (\Omega+z)$, we will impose $u_\eps \vee v_\eps = u_\eps$ and on $(\Omega+z) \setminus \Omega$, $u_\eps \wedge v_\eps = v_\eps$)
$$\int_\Omega  f(\varphi_\eps(\nabla u_\eps)) + \frac{(u_\eps-g)^2}{2} \ls \int_\Omega f(\varphi_\eps(\nabla (u_\eps \vee v_\eps))) + \frac{(u_\eps\vee v_\eps -g )^2}{2}$$
and
\begin{multline*}\int_{(\Omega+z)}  f(\varphi_\eps(\nabla v_\eps)) + \frac{(v_\eps(x)-(g(x-z)-(L-\delta)(\eps+\omega(\varphi_\eps^\circ(z)))))^2}{2} \\ \ls \int_{(\Omega+z)} f(\varphi_\eps(\nabla (u_\eps \wedge v_\eps))) + \frac{((u_\eps \wedge v_\eps)(x) -(g(x-z)-(L-\delta)\omega(\varphi_\eps^\circ(z))))^2}{2}.
\end{multline*}

We sum this two inequalities and notice that, since $u_\eps$ and $v_\eps$ are continuous, $$f(\varphi_\eps(\nabla u_\eps)) + f(\varphi_\eps(\nabla v_\eps)) = f(\varphi_\eps(\nabla (u_\eps \vee v_\eps))) + f(\varphi_\eps(\nabla (u_\eps \wedge v_\eps))),$$ that yields, denoting by $M_\eps$ the quantity $(L-\delta)(\eps+\omega(\varphi_\eps^\circ(z)))$,
$$0\ls \int_{ \Omega_z} ((u_\eps\vee v_\eps) - g)^2  - (u_\eps-g)^2 + ((u_\eps \wedge v_\eps)(x) -(g(x-z)-M_\eps))^2 - (v_\eps(x)-(g(x-z)-M_\eps))^2$$
which means
$$ 0 \ls \int_{ \Omega_z} -2 g (u_\eps \vee v_\eps) +2u_\eps g -2(g(x-z)-M_\eps) (u_\eps \wedge v_\eps) + 2(g(x-z)-M_\eps) v_\eps,$$
which is equivalent to
$$ 0 \ls \int_{\Omega_z} (u_\eps\vee v_\eps - u_\eps) \left( -g + g(x-z)-M_\eps \right). $$
Since $z \neq 0$, one has $\omega(z) >0$. In addition, $L-\delta >1$ so $$-g + g(x-z)-(L-\delta)(\eps+\omega(\varphi_\eps^\circ(z))) <0$$ ($g$ has $\varphi$ modulus of continuity $\omega$). Finally, $(u_\eps\vee v_\eps - u_\eps) \gs 0$ and since the integral is nonnegative, we must have $u_\eps\vee v_\eps - u_\eps = 0$ on the whole $\Omega_z$, which implies $u_\eps \gs v_\eps$ on $\Omega_z,$ that is
$$ u_\eps(x+z) - u_\eps(x) \ls (L-\delta) (\eps+\omega(\varphi_\eps^\circ(z))),$$
which is a contradiction with the definition of $L$.
\end{proof}

To conclude the proof of Theorem \ref{thmneu} , we just need to show that the constant $L$ can in fact not be reached on the boundary. We proceed by contradiction and assume that $\hat x \in \partial \Omega$.
\begin{rem}
In the isotropic case, this property is more easily shown. the boundary equation $\nabla u_\eps \cdot \nu$ ensures that the level lines of $u_\eps$ reach $\partial \Omega$ perpendicularly. The strict convexity of $\Omega$ prevent then the distance between two level-sets of $u_\eps$ from being attained on $\partial \Omega$, which means that the constant $L$ is not reached on the boundary.
\end{rem}

By assumption, we have, for every $x \in \overline{\Omega}$,
$$ \frac{u_\eps(x) - u_\eps(\hat y)}{\eps + \omega(\varphi_\eps^\circ (x-\hat y))} \ls \frac{u_\eps(\hat x) - u_\eps(\hat y)}{\eps + \omega(\varphi_\eps^\circ (\hat x- \hat y))},$$
which yields
$$u(x) \ls u(\hat y) +(\eps+\omega(\varphi_\eps^\circ(x-\hat y)))L =: \psi(x)$$
with equality at $x=\hat x$. Since $\hat x \neq \hat y$ and $\omega,\varphi^\circ$ are smooth away from zero, $\psi$ is an admissible test function for \eqref{eqanibound}, whose gradient does not vanish at $\hat x$. We can thus write
$$\underbrace{f'(\varphi(\nabla \psi))}_{>0 \text{ since }\nabla \psi \neq 0} \ \nabla \varphi (\nabla \psi) \cdot \nu  \ls 0.$$
Finally, one just has to notice that 
$$\nabla \psi(x) = L \omega'(\varphi^\circ(x-\hat y))\nabla \varphi^\circ (x-\hat y)$$
thus (since $\nabla \varphi$ is 0-homogeneous and $\varphi \nabla \varphi (\varphi^\circ (x) \nabla \varphi^\circ (x) ) = x$), 
$$\nabla \varphi(\nabla \psi ) (\hat x) = \nabla \varphi(\nabla \varphi^\circ (\hat x - \hat y)) = \frac{1}{\varphi^\circ (\hat x - \hat y) \varphi(\nabla \varphi^\circ (\hat x - \hat y))} (\hat x - \hat y)$$
which implies
$$(\hat x - \hat y) \cdot \nu \ls 0$$
which is not possible because of the uniform convexity of $\Omega.$

\section{A comparison result with a smooth set}
\label{onesmooth}
In this section, the anisotropy $\varphi$ is smooth on $\R^2\setminus\{0\}$ and $\varphi^2$ is strongly convex.
\begin{thm}Let $F$ and $G$ minimize in $\Omega$
$$\per_\varphi(F) + \int_F f $$ 
and
$$\per_\varphi (G) + \int_G g$$
with $f \ls g - \eps$.
We assume that $F \subset G$ and $\partial G$ is a $\mathcal C^1$ hypersurface. Then, either $F=G$ or $\partial F \cap \partial G = \emptyset.$
\label{thmsmooth}
\end{thm}
In what follows, for every $\tilde z \in \Rn$, we will denote by $\tilde z'$ the $n-1$ first component of $\tilde z$: $\tilde z= (\tilde z',\tilde z_n).$

\begin{rem}\begin{itemize}
\item When $f$ and $g$ are constant, we do not need $\eps$ to be positive (it can be zero).
\item If $F$ and $G$ are $\mathcal C^2$ surfaces, then the result is well known and relies on the classical strong maximum principle for elliptic equations. Indeed, the surfaces are locally graphs of functions that satisfy the prescribed mean curvature equation
$$- \div' \left( \nabla' \varphi(\nabla' u)\right) = f(x',u(x')),$$
which is known to be elliptic (see \cite{colding11} for details).
\item This theorem is already known when $f = g = 0$ in a more general version in \cite{solomon89} (in particular, the anisotropy can depend on the space variable, and the sets are only stationary, whereas they are minimizing in our framework).  Nonetheless, we present a simpler proof of this result, in the spirit of \cite{caffarelli93} (see also \cite{caffarelli10}).
\end{itemize}
 \end{rem}

We replace $\partial F$ by $\operatorname{supp}(D1_F)$ in order to work with a closed set. Let us assume that there exists $x_0 \in \partial F \cap \partial G$. We want to prove that it implies $F=G$. Since $\partial G$ is $\Cu$, $\partial G$ is the graph of some $\Cu$ function $\tilde v$ over $\tilde n^\perp$, with $\tilde n$ the outer normal to $G$ at $x_0$ (we may assume $x_0 = (0,\tilde v(0))$, $\tilde v(0) = 0$ and $v$ defined on $B'_{\tilde \rho}.$).

Thanks to Corollary \ref{cleanball}, for every $r$ sufficiently small, there exists a ball $B:=B_{qr}(x_r)$ of center $x_r$ and radius $qr$ with $B \subset  F \cap B_r(x_0).$ Since $\{\tilde z_n=0\}$ is tangent to $G$, $x_r$ must have a negative $n$-th component for $r$ small enough. Let $r_0$ satisfies this requirement and let $n= x_0 - x_{r_0}.$ Then, since $ \scal{n}{\tilde n} >0$, $\partial G$ is also a graph over $n^\perp$ of some function $v$ defined on $B'_\rho(0)$. Once again, we assume $v(0) = 0$ and denote $(z',z_n)$ the components of every $z \in \Rn$.
Then, we define
$$\forall |z'| \ls \rho, \quad u(z') := \sup \{z_n \in \R \, \vert \, (z',z_n) \in F\}.$$
Note that since $F \subset G$ and by definition of $u$, we must have $v \gs u$ on $B'_\rho.$

Moreover, $v$ is a $\mathcal C^1$ graph over $n^\perp$ on $B_\rho$, so it satisfies (in the variational sense, so also in the viscosity sense)
$$-\div' \left( \nabla' \varphi(\nabla' v,-1) \right) = g(x',v(x')).$$

\begin{prop}
 The function $u$ is upper semicontinuous and is a viscosity subsolution of 
\begin{equation} -\div' \left( \nabla' \varphi(\nabla' u,-1) \right) = f(x',u(x')). \label{viscss}\end{equation}
\end{prop}
\begin{proof}
 Let us first prove that $u$ is upper semicontinous. Let $x_n \to x \in B'_\rho.$ Then, we have a sequence $(x_k,u(x_k)) \in F$, which is bounded above. As a result, there exists a subsequence (still denoted by $(x_k,u(x_k))$) which converges (possibly $u(x_k) \to -\infty$). We want to show that $u(x) \gs \limsup_k u(x_k).$ If $u(x_k) \to - \infty$, nothing has to be done. If not, then $(x_k,u(x_k))$ is a converging sequence of $F$ which is closed. So, $(x,\lim u(x_k)) \in F$ and $u(x) \gs \limsup_k u(x_k).$ 

Now, let us show that it is a subsolution of \eqref{viscss}. Assume by contradiction that it is not. Then, there exist a smooth function $\psi$ and some $x_1 \in B'_\rho$ such that $u-\psi$ has a maximum at $x_1$ and
$$ -\div'\left( \nabla'\varphi'(\nabla' \psi,-1) \right) > f.$$
On can assume that $x_1 = 0$ and $u(x_1) = \psi(x_1)$ and that the maximum is strict. Let $\Gamma$ be the graph of $\psi.$ We want to generalize the result by Caffarelli and Cordoba \cite{caffarelli93}, which showed that the signed distance to an area minimizing hypersurface is superharmonic. To this aim, we work with the $\varphi$-relative distance 
$$d_\varphi(x,y) = \varphi^\circ (x-y) \quad \text{where} \quad \varphi^\circ (x) = \sup_{\varphi(\nu) \ls 1} \scalv{x}{\nu}$$ and
$$d^\Gamma_\varphi(x)=\inf\{d_\varphi(x,y) \, \mid \, y\in \Gamma\}.$$
Then, we defined the signed $\varphi$-relative distance to $\Gamma$ by setting
$$d(x',x_n) = d^\Gamma_\varphi(x',x_n) 1_{\{x_n \ls \psi(x)\}} - d^\Gamma_\varphi(x',x_n) 1_{\{x_n \gs \psi(x')\}}.$$
Since $\Gamma$ and $\varphi$ are smooth, there exists a tubular neighborhood of $\Gamma$ where $d$ is smooth. 

\begin{lem}
We have
 \begin{equation}
 -\div'(\nabla'\varphi(\nabla'\psi,-1))(0) = -\div(\nabla \varphi (\nabla d)) (0,0). \label{lapdist}
\end{equation}
\end{lem}
\begin{proof}
Let us first notice that $d (x',\psi(x')) = 0$, so that $\nabla' d + \partial_n d \nabla' \psi = 0$. Hence, since $\nabla \varphi$ is $0$-homogeneous and even, we get
$$\frac{\partial \varphi}{\partial x_i}(\nabla' \psi, -1)=\frac{\partial \varphi}{\partial x_i}\left(-\frac{\nabla' d}{\dpa{d}{n}} , -1 \right) = \frac{\partial \varphi}{\partial x_i}\left(\nabla' d,\dpa{d}{n}\right) = \frac{\partial \varphi}{\partial x_i}(\nabla d(x',\psi(x')).$$
As a matter of fact,
\begin{align*}
\div'(\nabla'\varphi(\nabla' \psi,-1)) &= \div'(\nabla' \varphi(\nabla d(x',\psi(x')) = \sum_{i=1}^{n-1} \frac{\partial}{\partial x_i} \left(\dpa{\varphi(\nabla d(x',\psi(x')))}{i} \right) \\
&= \sum_{i=1}^{n-1} \sum_{j=1}^n \dpp{\varphi}{i}{j} \frac{\partial }{\partial x_i} \left( \dpa{d(x',\psi(x'))}{j} \right) \\
&= \sum_{i=1}^{n-1} \sum_{j=1}^n \dpp{\varphi}{i}{j} \left( \dpp{d}{i}{j} + \dpp{d}{j}{n} \dpa{\psi}{i} \right).
\end{align*}
As a result,
\begin{align*}
\div(\nabla \varphi(\nabla d) &= \sum_{i=1}^n \frac{\partial }{\partial x_i} \left( \dpa{\varphi(\nabla d)}{i} \right) = \sum_{i,j=1}^n \dpp{\varphi(\nabla d)}{i}{j} \dpp{d}{i}{j} \\
&=\sum_{i=1}^{n-1}\sum_{j=1}^n \dpp{\varphi(\nabla d)}{i}{j} \dpp{d}{i}{j} + \sum_{j=1}^n \dpp{\varphi}{n}{j} \dpp{d}{n}{j} \\
&= \div'(\nabla'\varphi(\nabla' \psi,-1)) + \sum_{j=1}^n \dpp{\varphi}{n}{j} \dpp{d}{n}{j} - \sum_{i=1}^{n-1} \sum_{j=1}^n \dpp{\varphi}{i}{j} \dpp{d}{n}{j} \dpa{\psi}{i} \\
&=\div'(\nabla'\varphi(\nabla' \psi,-1)) + \sum_{j=1}^n \dpp{d}{n}{j} \left( \sum_{i=1}^{n-1} \dpp{\varphi}{i}{j} \dpa{\psi}{i} - \dpp{\varphi}{n}{j} \right) \\
&= \div'(\nabla'\varphi(\nabla' \psi,-1)) - \sum_{j=1}^n \dpp{d}{n}{j} \left(  \sum_{i=1}^{n-1} \dpp{\varphi}{i}{j} \frac{\dpa{d}{i}}{\dpa{d}{n}} + \dpp{\varphi}{n}{j} \right) \\
&=\div'(\nabla'\varphi(\nabla' \psi,-1)) - \sum_{j=1}^n \dpp{d}{n}{j} \frac{1}{\dpa{d}{n}} \sum_{i=1}^n \dpp{\varphi}{i}{j} \dpa{d}{i}.
\end{align*}

Let us show that the last term of the last equality vanishes. Indeed, one has $$\varphi^\circ(\nabla \varphi(\nabla d)) =1,$$ whose derivative provides
\begin{equation}\forall i \ls n,\quad \sum_{j=1}^n \dpa{\varphi^\circ(\nabla \varphi(\nabla d))}{j} \cdot \frac{\partial }{\partial x_i} \left( \dpa{\varphi(\nabla d)}{j}\right) = 0.\label{phiphic}\end{equation}
In addition, thanks to the equality (which holds for any anisotropy) $(\varphi^\circ \nabla \varphi^\circ) (\varphi(\xi) \nabla \varphi(\xi))=\xi$, one obtains $\nabla \varphi^\circ (\nabla \varphi(\xi)) = \frac{\xi}{\varphi(\xi)}$. 
Then, \eqref{phiphic} can be rewritten
$$\forall i,\quad \sum_{j,k=1}^n \frac{1}{\varphi(\nabla d)}\dpa{d}{j} \dpp{d}{i}{k} \dpp{\varphi}{j}{k} =0,$$
which implies for $i=n$ (and some changes of indices)
$$\forall i, \quad \sum_{j=1}^n \dpp{d}{n}{j} \left(\sum_{i=1}^n \dpa{d}{i} \dpp{\varphi}{i}{j} \right)=0,$$
what was expected.
\end{proof}

Let $\delta$ be small and fixed. Let $\tilde \Omega$ be the epigraph of $\psi$ (we have $\partial \tilde \Omega = \Gamma).$ We are interested in $(\tilde \Omega - \delta e_n)\cap B'_\rho$. Then, if $\delta$ is small enough,
\begin{itemize}
 \item $F \setminus (\tilde \Omega - \delta e_n) \cap B'_\rho$ is a compact perturbation of $ F$ in $B_\rho$ (the maximum is strict).
\item $(\tilde \Omega - \delta e_n) \cap F$ has a nonempty interior in $F$ (clean ball property).
\item If $\tilde d = d(\cdot + \delta e_n)$, we have $-\div(\nabla \varphi(\nabla \tilde d)) \gs f+\eta$ ($\eta >0$) in $(\tilde \Omega - \delta e_n) \cap F$ (continuity of $d$ and \eqref{lapdist}).
\end{itemize}
Let $\Omega = (\tilde \Omega- \delta e_n) \cap F,$. If $F$ were $\mathcal C^2$, we would have
$$
 \int_\Omega -\div(\nabla \varphi(\nabla \tilde d)) = -\int_{(\Gamma - \delta e_n)\cap F} (\nabla \varphi (\nabla \tilde d)) \cdot \nu_\Omega d\sigma - \int_{\partial F \cap \Omega } (\nabla \varphi(\nabla \tilde d)) \cdot \nu_\Omega d\sigma
$$
which yields, using $-\div(\nabla \varphi(\nabla \tilde d)) \gs f+\eta$ and noting that on $\partial F$, $\nu_\Omega = \nu_F$, we obtain
$$
 -\int_{\Gamma \cap (F + \delta e_n)} (\nabla \varphi (\nabla d)) \cdot \nu_\Gamma d\sigma + \int_{\partial F \cap \Omega} (\nabla \varphi(\nabla \tilde d)) \cdot \nu_F d\sigma \gs \int_\Omega f+\eta.
$$
Recall that $F$ is minimizing, we can also write (comparing $F$ to the compact perturbation $F \setminus \Omega$)
$$ \int_\Omega f \ls \int_{\Gamma \cap (F + \delta e_n)} \varphi(\nu) d\sigma - \int_{\partial F \cap \Omega} \varphi(\nu) d\sigma.$$
Substracting the second inequality to the first one, we obtain
\begin{multline*} \int_\Omega \eta \ls -\int_{\Gamma \cap (F + \delta e_n)} \varphi(\nu) d \sigma + \int_{\partial F \cap \Omega} \varphi(\nu) d\sigma  \\ -\int_{\Gamma \cap (F + \delta e_n)} (\nabla \varphi (\nabla d)) \cdot \nu_\Gamma d\sigma + \int_{\partial F \cap \Omega} (\nabla \varphi(\nabla \tilde d)) \cdot \nu_F d\sigma.
\end{multline*}
Now, note that on $\Gamma$, we have $\nabla d = \frac{\nu}{\varphi(\nu)}$. On the other hand, $\nabla \varphi (\nu) \cdot \nu = \varphi(\nu)$ (because of the homogeneity of $\varphi$), which implies $\nabla \varphi(\nabla d) = \varphi(\nu)$ on $\Gamma$. We can then compute
$$ \int_{\Gamma \cap (F + \delta e_n)} (\nabla \varphi (\nabla d)) \cdot \nu_\Gamma d\sigma = -\int_{\Gamma \cap (F + \delta e_n)} \varphi(\nu) d\sigma.$$
In addition, since $\varphi^\circ (\nabla \varphi (\nabla d)) = 1$, we also have $\nabla \varphi (\nabla d)\cdot \nu \ls \varphi(\nu).$ That implies
$$\left \vert \int_{\partial F \cap \Omega} (\nabla \varphi(\nabla \tilde d)) \cdot \nu_F \right \vert \ls \int_{\partial F \cap \Omega} \varphi(\nu_F)d\sigma.$$

These two relations yield
$$\int_\Omega \eta \ls 0$$
which is not possible.

If $F$ is not regular, we select a sequence of $F_n \to F$ with $F_n$ regular and $1_{F_n} \to 1_F$ in $BV$ and we reproduce this construction on $F_n$ and pass to the limit (note that $\eta$ does not depend on $n$).
\end{proof}

At this stage, $u$ is a viscosity subsolution of
$$-\div'\left( \nabla'\varphi(\nabla' u,-1) \right) = f(x',u(x'))$$ whereas $v$ is a viscosity supersolution of 
$$-\div'\left( \nabla'\varphi(\nabla' v,-1) \right) = g(x',v(x')) \gs f(x',u(x')).$$
So, $v$ is also a supersolution of \eqref{viscss}. We also know that $v\gs u.$ We would like to prove that $v > u$, because that would ensure that $\partial F \cap \partial G = \emptyset.$ So, we need a strict comparison principle for viscosity solutions. This is found in \cite{giga05}. Let us check that the assumptions are fullfiled. This article deals with an equation written as (see \cite[Remark 3.6]{giga05} for the right hand side)
$$ F(Du,D^2u) = h $$
with $F$ satisfying
\begin{enumerate}
 \item The function $F : \Rn \times \mathcal S_n \to \R$ is continuous,
\item There exists a coercive function $w$ such that for all $p,X,Y$,
$$ F(p,X) - F(p,Y) \gs w(p,X-Y),$$
\item For every $M,K >0$ and $ |q|,|\tilde q| \ls K$, $\Vert X \Vert \ls M$, one has
$$|F(q,X) - F(\tilde q, X) | \ls L_{M,K} |q-\tilde q|.$$
\end{enumerate}
Here, we have 
$$F(p,X)= \sum_{i,k=1}^{n-1} \dpp{\varphi}{i}{k} (p,-1) X_{ik} = \tr(D'^2 \varphi(p,-1) X).$$
It is clearly continous.
\begin{itemize}
 \item If $p,q \in \Rn$ such that $|p|,|q| \ls M$, if $X \in \mathcal S^n$ satisfies $|X| \ls K$, one obtains
$$|F(p,X)-F(q,X)| =  \left \vert \sum_{i,k} \left( \dpp{\varphi}{i}{k} (p,-1) - \dpp{\varphi}{i}{k} (q,-1) \right) X_{ik} \right \vert \ls L_{M,K} |p-q|.$$
\item Let $p \in \Rn$ with $|p|\ls M$ and $X,Y \in \mathcal S_n$ such that $X \ls Y$.\\
The assumption on $\varphi$ imply that $p \mapsto \varphi(p,-1)$ is uniformly convex with constant $\lambda(M)$ on every $\{|p| \ls M\}$ (see the proposition below)
As a result, one has
$$\lambda(M) \tr(Y-X) \ls F(p,X)-F(p,Y)=\sum_{i=1}^{n-1} \dpc{\varphi}{i}(p,-1) \lambda_i \ls \Lambda \tr(Y-X)$$
with $\Lambda$ is the maximum of the spectral radius of $D^2\varphi^2(q)$ for $q=1$.
\end{itemize}
Hence, \cite[Th. 3.1]{giga05} applies and gives the following alternative: either $u=v$ on $B_\rho$ or $u<v$. That is exactly Theorem \ref{thmsmooth}.

Finally, note that in the framework of \eqref{arof}, we have $f <g - \eps$ so $F$ and $G$ cannot coincide.

During the proof, we used the
\begin{prop}
The function $\tilde \varphi : p \mapsto \varphi(p,-1)$ is uniformly convex on $\{|p| \ls M \}$, with a constant $\lambda(M).$
\end{prop}
\begin{proof}
First, recall a few properties of the anisotropy $\varphi.$ By assumption, the sets $\{\varphi \ls t\}$ (Wulff shape of radius $t$) are homothetic convex subsets which contain a neighborhood of zero. In addition,  $D^2\varphi^2 \gs \alpha I.$ Noticing that
$$ D^2 \varphi = \frac{1}{\varphi} D^2 \varphi^2 - \frac{1}{\varphi} \nabla \varphi \otimes \nabla \varphi,$$
we see that $D^2 \varphi$ is positive definite on $T(p,-1)$, the tangent plane to the Wulff shape $\{\varphi = \varphi(p,-1)\}$ at $(p,-1),$ with eigenvalues bigger than $\frac{\alpha}{\varphi(p,-1)}.$\\
Finally, 

Since $\varphi$ is smooth around $(p,-1)$, to prove the proposition, we only have to control the eigenvalues of $D^2 \tilde \varphi (p) = \left. D^2 \varphi (p,-1) \right \vert_{\{(e,0)\}}$. Let us write $e = e^T + e^0$ the decomposition of $e$ with respect to $\nabla \varphi(p,-1)^\perp$ and $\operatorname{span}(p,-1)$ (note that this projection is not orthogonal). Then,
$$D^2 \varphi(p,-1) \cdot (e,e) = \underbrace{D^2 \varphi(p,-1) \cdot (e^T,e^T)}_{\gs \alpha |e^T|^2} + \underbrace{2D^2 \varphi(p,-1) \cdot (e^0,e^T) + D^2 \varphi(p,-1) \cdot (e^0,e^0)}_{=0 \text{ since } D^2 \varphi(p,-1) \cdot e^0 = 0}.$$
To conclude, we need to show that there exists a constant $\gamma(M)$ such that $|(e,0)^T| \gs \gamma(M) |(e,0)|$ as soon as $|p| \ls M$. Since there is an angle between $(e,0)$ and $(p,-1)$ which remains far from 0 on $\{|p| \ls M\}$, this is equivalent to show that the norm of the projection is controlled, or to show that the angle between $(p,-1)$ and $\nabla \varphi(p,-1)$ remains far from $\frac \pi 2.$ This is true using that the Wulff shape is a convex set which contains a neighborhood of zero.

Finally, $D^2 \varphi$ is uniformly convex with constant $\frac{\alpha \gamma(M)}{\beta(M)}$ where $\beta(M) = \min\limits_{|p| \ls M} \varphi(p,-1).$

\end{proof}

\section{A result on mean convex domains with Dirichlet conditions}
\label{secjer}

In this section, we link the minimizer $u$ to the image $g$ using Dirichlet conditions on the boundary of the domain in the spirit of a recent work \cite{jerrard13} (see also \cite{Mar83} and \cite{Ray93} for previous works in this direction). To give the assumptions on $\Omega$, we need the 
\begin{defi}
Let $\varphi$ be a norm in $\Rn$. We say that $\Omega$ satisfies the barrier condition if for every $x_0 \in \partial \Omega$ and $\eps >0$ sufficiently small, if $V$ minimizes $\per_\varphi$ in
$$\{W \subset \Omega \, \mid \, W \setminus B_\eps(x_0) = \Omega \setminus B_\eps(x_0)\},$$
then 
$$\partial V^{(1)} \cap \partial \Omega \cap B_\eps(x_0) = \emptyset.$$
\label{barjer}
\end{defi}
\begin{rem}
The barrier condition means that $\partial \Omega$ is not a local minimizer of the perimeter (there is always a inside perturbation of $\Omega$ which provides a set with strictly smaller perimeter). Note that if $\varphi$ is the Euclidean norm and $\Omega$ is smooth, this property is the strict mean-convexity of $\Omega.$ (positive mean curvature).
\end{rem}

\begin{thm}
 Let $\varphi$ be a norm in $\Rn$ which is $\mathcal C^2$ in $\Rn \setminus \{0\}$ and such that $\varphi^2$ is strongly convex. Let also $\Omega$ be a bounded Lipschitz open subset which satisfies the barrier condition and $g$ be continuous on $\partial \Omega$. Then, there is a unique minimizer $u$ of 
\begin{equation} u = \argmin_{\substack{v \in BV \\ v = g \text{ on } \partial \Omega}} \int_{\Omega} \varphi(\nabla u)\label{eqjer}\end{equation}
where the equality $v=g$ on $\partial \Omega$ means, as in \cite{jerrard13}, that 
\begin{equation} \forall x \in \partial \Omega , \; \lim_{r \to 0} \esup_{\substack{y \in \Omega \\ |x-y| \ls r}} |v(y) - g(x)| = 0.\label{eqconb}\end{equation}
In addition, this minimizer is continuous.
\label{thmjer}
\end{thm}

\begin{rem}
Since $\varphi$ is not strictly convex as in \cite{miranda65} (because of the homogeneity), we have to find another way to obtain something similar to \cite[Th. 2.1]{miranda65}. 
This is done in \cite{jerrard13}, which provides such proposition in the case we deal with. Proceeding as in \cite{miranda65}, we could directly complete the proof (note that due to the space dependency, Jerrard et al. can obtain continuity of the minimizer only in dimension $\ls 3$, using the regularity of the level-sets of $u$: see \cite{SchSimAlm77} for details). Nonetheless, since we can take advantage of the translation invariance of the minimizers (which does not exists in \cite{jerrard13} because of the space dependency), we give a much simpler proof of the continuity of $u$. In particular, we will use no deep results neither on topological dimension nor on connected components of regular points of a minimal surface.
\end{rem}
For simplicity, we assume that $g$ is defined and continuous on the whole $\Rn.$

We first recall the proof of the existence part of the theorem (it is already done in \cite{jerrard13}). Let $u$ be a minimizer of \eqref{eqjer} in the class
$$\mathcal A_f := \{v \in BV(\Rn) \, \vert \, v = g \text{ on } \Omega^c \}.$$
It exists by standard techniques of calculus of variation. 

We recall that thanks to the coarea formula (used similarly as in Proposition \ref{proparof}), the level sets $\et$ minimize
\begin{equation}\et = \argmin_{E \setminus \Omega = F_t \setminus \Omega} \per_\varphi(E),\label{eqperjer}\end{equation}
with $F_t  := \{g >t\},$ where the exponent $(1)$, as before, means that we consider the subset of points with density one:
$$F^{(1)} := \left \{ x \in \Omega\, \middle | \, \lim\limits_{r \to 0}  \frac{|B_r(x) \cap F|}{|B_r|} = 1 \right \}.$$
We recall that thanks to Proposition \ref{propdensone}, $\et$ are open subsets.

To show that \eqref{eqconb} is in fact satisfied by $u$, we recall the following lemma (which is simply a restatement of \cite[Th. 1.1]{jerrard13}).
\begin{lem}
Let $\hat x \in \partial \Omega$ and let $t$ and $\eps$ such that $g(\hat x) \ls t-\eps$. Then, there exists $\rho >0$ such that
$$\et \cap B_\rho(\hat x) = \emptyset.$$ 
The same result holds for $g(\hat x) \gs t+\eps$.
\label{lembord}
\end{lem}
Now, let $u$ be a minimizer of \eqref{eqjer}. We prove that it is continuous. We will show that its level sets $E_t$ and $E_s$, for $s < t$, satisfy $E_t^{(1)}  \Subset E_s^{(1)}$. 

We begin by noting that these two sets cannot touch near $\partial \Omega$.
\begin{lem}
 Let $s < t.$ There exists $\delta >0$ and $\eps >0$ such that for every $x\in \Omega \cap \partial E_s^{(1)}$ with $d(x,\partial \Omega) \ls \delta$ and $y \in \et \cap \Omega$, then $d(x,y) \gs \eps.$ \label{paobord}
\end{lem}
This is straightforward using Lemma \ref{lembord}, with $\eps = \frac{t-s}{2}.$ The compactness of $\partial \Omega$ provides the expected $\delta.$

Before proving Theorem \ref{thmjer}, we state a very standard but useful 
\begin{lem}
 \label{twomin}
Let $E$ and $\tilde E$ be two minimizers of \eqref{eqperjer} with $F_t$ replaced respectively by $F$ and $\tilde F$ and assume that $(E \cup \tilde E) \setminus E$ is a compact subset of $\Omega$. Then, $E \cup \tilde E$ and $E \cap \tilde E$ are minimizers of \eqref{eqperjer} with $F_t$ replaced respectively with $F$ and $\tilde F$.
\end{lem}
\begin{proof}
 The proof is also very standard but we give it for completeness. We notice that $(E \cup \tilde E) \setminus \Omega = E \setminus \Omega =F \setminus \Omega$ so $E \cup \tilde E$ is an admissible perturbation for $E$. One therefore can write
$$ \per_\varphi(E\cup \tilde E) \gs \per_\varphi E.$$
Similarly $E \cap \tilde E$ is an admissible perturbation for $\tilde E$ and we can write,
$$\per_\varphi(E\cap \tilde E) \gs \per_\varphi \tilde E.$$
By summing the two inequalities and recalling \eqref{addperv}, we must have equality in the inequalities. That is the claim.
\end{proof}

\begin{proof}[Proof of Theorem \ref{thmjer}.]
We proceed by contradiction. Let us assume that there exists $x_0 \in \partial \es \cap \partial \et$ and let $r_0 = \frac{min(\delta,\eps)}{10},$ where $\delta$ and $\eps$ are the constants provided by Lemma \ref{paobord}. Thanks to this lemma, $d(x_0,\partial \Omega) \gs \delta.$ 

Recalling that $\partial \es$ and $\partial \et$ are regular up to a compact set of dimension at most $n-3$ we can choose $\alpha \in \partial \es$ and $\beta \in \partial \et$ two regular points such that 
$$|\alpha - x_0| \ls r_0 \quad \text{and} \quad |\beta - x_0| \ls r_0.$$
If $\nu = \alpha - \beta$, note that $|\nu| \ls \frac 12 \min(\delta,\eps)$ thus $\es$ and $\et + \nu$ do not touch near the boundary $\partial \Omega.$

\paragraph{The regular set $\reg(\partial \es \cap B_{r_0/2} (x_0))$ is a set of pieces of parallel hyperplanes.}
The point $\alpha$ is regular means that one can find a direction $n$ such that both $\partial \es$ and $\partial \tilde \et := \partial \et + \tilde \nu$ are ($\mathcal C^2$) graphs around $\alpha$. Since $\tilde \et \cap \es$ and $\tilde \et \cup \es$ are also minimizers (thanks to Lemma \ref{twomin}) and are both graphs around $\alpha$, we have two functions $w_1 \ls w_2$ such that $w_1(\alpha') = w_2(\alpha')$ and which satisfy the zero $\varphi$-mean curvature equation for graphs
$$ \div' \left( \nabla' \varphi(\nabla' w_i,-1) \right) = 0.$$
By comparison principle for graphs (\cite{giga05}, the one used in Section \ref{onesmooth}), they must coincide locally. 

Notice that this coincidence is true for every pair of regular points $\alpha,\beta \in B_{r_0/2} (x_0)$ with $\alpha \in \partial \es$ and $\beta \in  \partial \et + \nu$. Leaving $\beta$ and moving only $\alpha$, this proves that every regular point $\alpha$ of $\partial \es \cap B_{r_0/2}(x_0)$ has a neighborhood (in $\partial \es$) which coincides with a neighborhood of $\beta$ in $\partial \et + \nu.$ As a result, every regular point of $\partial \es \cap B_{r_0/2}(x_0)$ has the same normal (let us call it $\omega$). Since in addition, the set of regular points is an open subset of $\partial \es$, the connected components of $\reg(\partial \es)$ are pieces of affine hyperplanes parallel to $\omega^\perp,$ oriented either by $\omega$ or by $-\omega$.

Of course, $\reg((\partial \et) \cap B_{r_0/2} (x_0))$ satisfies the same property.

\paragraph{These pieces of hyperplans which cross $B_{r_0/4} (x_0)$ fill $B_{r_0/4} (x_0)$.} Indeed, Let $x \in \reg \partial \es \cap B_{r_0/4} (x_0)$. Then, there is a ball $\hat B$ (of radius $\hat r$) around $x$ such that $\partial \es \cap \hat B$ is exactly a diameter of $\hat B$. Let us assume that the normal of $\partial \es$ is $\omega$ in $\hat B$. Then, let us consider the cylinder $\hat C$ generated by $\hat B$ and a vector $e \perp \omega$ in the ball $B_{r_0/2} (x_0)$. One can write, for every $R$ such that $z+Re \in B_{r_0/4} (x_0)$
$$ \int_{\substack{z \in e^\perp \\ |z| \ls \hat r}}\int_{\tau = 0}^{Re} |D(\chi_{\es} (z+\tau e))| = \int_{\hat C \cap \partial \es} |\nu_{\es} \cdot e|dH^{n-1} = 0 $$
because $e \perp \omega.$
Then, for almost every $z \in e^\perp$ with $|z| \ls \hat r$, we have $\tau \mapsto \chi_{\es}(z+\tau e)$ is constant. That means that if $z + \tau e$ belongs to $\es$ for some $\tau$, that is true for every $\tau$ (and similarly for $\notin \es$). Finally, the piece of hyperplane of $\reg \es$ which is a diameter of $\hat B$ exists in the whole cylinder $\hat C$, and since $e$ is arbitrary in $\omega^\perp$, in the whole ball $B_{r_0/4} (x_0)$ (we have to stay sufficiently far from $\partial B_{r_0 /2} (x_0)$ in order to keep the whole cylinder inside $B_{r_0 /2} (x_0)$).

\paragraph{The point $x_0$ is in fact regular} Thanks to the previous paragraphs, $\reg \es \cap B_{r_0/4}(x_0)$ is a (finite, for measurability reasons) set of hyperplanes.

In addition, since $x_0 \in \partial \es \cap (\partial \et + \nu)$, we have a sequence of points in $\reg \es$ (which therefore belong to hyperplanes) which converge to $x_0$. Using the finiteness of the set of hyperplanes, $x_0$ must be in one of them. So, $x_0$ is in fact a regular point of $\es$ (the same holds for $\et + \nu$), and $\es$ and $\et + \nu$ coicinde around $x_0$. That is exactly saying that $\partial \es \cap (\partial \et + \nu)$ is open in $\partial \es.$ It is closed by definition. To reach a contradiction, we now need to show that every connected component of $\partial \es$ has to reach the boundary $\partial \Omega$. 
\end{proof}
\begin{prop}
There is no connected component of $\partial \es$ which is compact in $\Omega$.
\end{prop}
\begin{proof}
 Let us proceed by contradiction and call $\Gamma$ a compact connected component of $\partial \es.$ We denote by $\delta$ the distance between $\Gamma$ and $\partial \Omega$. One can find a continuous function $f : \partial \es \to \{0,1\}$ which is $0$ on $\Gamma$ and $1$ on $\es \setminus \{\dist(x,\Gamma) < \delta/2\}.$ Since $\es$ is compact, $f$ is uniformly continuous. Let call $\omega$ its modulus of continuity and extend $f$ to the whole $\Omega$ by
 $$ f(x) = \sup_{y \in \es} f(y) + \omega(x-y).$$
 In addition, we may assume that $f \gs 1$ on $\partial \Omega$ (possibly replacing $f$ by 
 $\max(f,1-\dist(x,\partial \Omega)/\delta)$).
 Note that $f(x)=\alpha \in (0,1)$ implies that $x$ remains far from $\partial \es.$
 
 Now, let us introduce $C$ as the connected component of the open subset $\{f < \frac 14\}$ which contains $\Gamma$ and set
 $$a := \min_{x \in \partial C} u^+(x) \quad \text{and} \quad b := \max_{x \in \partial C} u^-(x)$$
 where $u^\pm$ are defined in Definition \ref{defupm}.
 If $a > s$, then we define $v$ such that $v=u$ everywhere but in $C \cap \left \{u \ls \frac{a+s}{2} \right\}$ where we set $v = \frac{a+s}{2}.$ \\
 Then, we notice that $v$ differs from $u$ only in a neighborhood of $\Gamma$ and 
 $$ \int_C \varphi \left(\nabla \left( u \vee \frac{a+s}{2} \right) \right) + \int_C \varphi \left( \nabla  \left( u \wedge \frac{a+s}{2} \right) \right) \ls \int_C \varphi(\nabla u) + \int_C \underbrace{\nabla \left( \frac{a+s}{2} \right)}_{=0}.$$
 Then, $v$ is also a minimizer with 
 $$(\partial \{v >s\}) \cap C = \emptyset,$$
 which implies
 $$\per(\{v > s\}) < \per(\{u>s\}),$$
 which cannot happen.
 
 Similarly, if $b \ls s$, then we introduce $v = u \wedge s$ in $C$, $v=u$ in $C^c$ and we also reach a contradiction.
 
 Finally, we cannot have either $\partial C \subset \{u > s\}$ or $\partial C \subset \{u \ls s\}.$ But on the other hand, we have $\partial C \subset \{f=\frac 14\}$ which means that $\partial C$ cannot be too close to $\partial \es$: this is a contradiction.
 
\end{proof}

\begin{rem}
 All the proof above can be reproduce with $\et =\{u > t \}$ and $\es = \{v > s\}$, if $u$ and $v$ are two minimizers: that shows $u = v$ a.e.
\end{rem}

\section{Local continuity}
\label{seclocal}
In this section, we get back to the isotropic case \eqref{rof}. We want to prove the
\begin{thm}
 Let $g : \Omega \to \R$ be continuous and bounded and let $u$ be a minimizer of
\begin{equation}
 \int_\Omega |Du| + \int_\Omega \frac{(u-g)^2}{2}. \label{rof1}
\end{equation}
Then, $u$ is continuous.
\end{thm}
Note that this theorem is local and therefore extends \cite[Th. 2]{chambolle11} (but for continuous functions only).

We will use the level sets. More precisely, let $E_s \subset E_t$ two level sets of $u$ (with $s > t$). We know that they minimize respectively (with respect to compact perturbations in $\Omega$)
$$\per(E,\Omega) + \int_{E \cap \Omega} s-g$$ and
$$\per(E,\Omega) + \int_{E \cap \Omega} t-g.$$

The strategy is the following. We know that two minimal surfaces satisfy a strict comparison principle \cite{simon87}, and we can extend this proof to constant mean curvature surfaces. As a result, we first show that we can create two different constant mean curvature which stands between $E_s$ and $E_t$. Then, we show that these surfaces do not touch. So, neither can $E_s$ and $E_t$. As before, we replace $E_s$ and $E_t$ by the set of points of density one.

\subsection{Back to constant mean curvature}
We assume (and we hope that we can get a contradiction) that there is $x_0 \in \partial E_s \cap \partial E_t.$ Note first that since $E_s$ and $E_t$ have mean curvature which are different, they cannot coincide on a neighborhood of $x_0$. By continuity of $g$, we can find $\rho >0$ such that on $B_\rho(x_0)$, we have $g(x_0) -\alpha < g(x) <g(x_0)+ \alpha$ with $\alpha:= \frac{s-t}{100}.$
So, let $a = s- g(x_0) - \alpha$. Then, 
$$\forall x,y \in B_\rho(x_0), \quad t-g(x) \ls a \ls s - g(y) .$$

Now, we introduce $E$ with finite perimeter in $\Omega$ as the minimizer of
$$ E = \argmin_{G \Delta E_s \subset B_\rho(x_0)} \per(G,\Omega) + a |G \cap B_\rho(x_0)| $$
and similarly, $F$ with finite perimeter in $\Omega$ and minimizing
$$ F = \argmin_{G \Delta E_t \subset B_\rho(x_0)} \per(G,\Omega) + a |G \cap B_\rho(x_0)| .$$
Note that $E$ and $F$ have variational constant mean curvature $a$.

Using the standard (weak) comparison principle, we have $E_s \subset E \subset F \subset E_t.$ In addition, since $E_s$ and $E_t$ cannot coincide, $E$ and $F$ cannot either. On the other hand, we must have $x_0 \in \partial E \cap \partial F$.

To show that $\partial E_s$ and $\partial E_t$ cannot touch, it is enough to prove that $\partial E \cap \partial F = \emptyset$. That is to prove the

\begin{thm}
Let $a \in \R$ and $E \subset F$ such that $E$ and $F$ both minimize (with respect to compact perturbations) in an open subset $O$,
\begin{equation}\per(E;O) + a |E \cap O| . \label{minp}\end{equation}
Then, either $E = F$ or $\partial E \cap \partial F = \emptyset$.
\label{simonp}
\end{thm}
This theorem is known for $a=0$ (see \cite{simon87}) and the general proof is really similar to \cite{simon87}. Nonetheless, almost every step of the proof should be slightly modified so we prefer giving a full and self contained proof of Theorem \ref{simonp}, recalling some properties on the minimizers that are known but whose proof are often splitted into different papers. 

In what follows, we take $O = \Omega$ (we can reduce the latter since we only want a local result).

\subsection{Properties of minimizers}
Before proving Theorem \ref{simonp}, we first recall results on minimizers of \eqref{minp} that will be crucial in the proof. These results can be found in \cite{giusti84,ambrosio00} for $a=0$ (see also \cite{Mag00}) and in \cite{massari74,Massari75,massari752} for prescribed curvature in $L^p$. Since the last papers are more technical that what we need for constant curvature, we chose to give the proofs. We begin by the usual monotonicity formula (see \cite{massari752})

\begin{prop}[Monotonicity formula]
 Let $E$ be a minimizer of \eqref{minp}. Then, for every $s < r$ and every $x \in \partial E$, we have
$$\frac{\per(E,B_r(x))}{r^{n-1}} - \frac{\per(E,B_s(x))}{s^{n-1}} \gs - (n-1) \omega_n |a| (r-s). $$
\label{propmon}
\end{prop}
\begin{rem}
That formula explains why we restricted ourselves to the isotropic case. In the anisotropic non Riemannian case, this formula is known not to hold \cite{All74}.
\end{rem}

\begin{cor}
For all $x \in \partial E$ and $\dist(x,\partial \Omega) > r > 0$ we have
\begin{equation}
 r^{1-n} \int_{B_r(x)} |D1_E| \gs \omega_{n-1} - (n-1)\omega_n |a| r.
\label{lbper}
\end{equation}
\label{corlbper}
\end{cor}

\begin{lem}{\cite[Th. 2]{Massari75}}
 Let $(E_\lambda)$ be a family of minimizers of \eqref{minp} with $a_\lambda$ ($\in \R$) instead of $a$, and let $E$ minimize \eqref{minp}. We assume that $E_\lambda  \to E$ in $\llo$ and that $a_\lambda \to a$. Then, for every bounded set $D$ (with Lipschitz boundary) such that
$$\int_{\partial D} |D1_E| = 0,$$
we have
$$\lim_\lambda \int_{\overline D} |D1_{E_\lambda}| = \int_{\overline D} |D1_E|.$$
\label{limTV}
\end{lem}

The following theorem, usually called \emph{improvement of flatness}, is the key result in the regularity proof. It can be found in \cite{massari752}.
\begin{thm}[De Giorgi]
 Let $E$ minimize \eqref{minp} and $\alpha \in (0,1)$. Then, there exits a constant $\sigma(n,\alpha,|a|)$ such that for all $\eta \ls \sigma$ and $r \ls \eta^2$, if $E$ satisfies
$$\int_{\overline B_r(x)} |D1_E| - \left \vert \int_{\overline B_r(x)} D1_E \right \vert \ls \eta r^{n-1},$$
then, we have
$$\int_{\overline B_{\alpha r}(x)} |D1_E| - \left \vert \int_{\overline B_{\alpha r}(x)} D1_E \right \vert \ls  \alpha^{1/2} \eta (\alpha r)^{n-1}.$$
\label{flatness}
\end{thm}

\subsection{blowups}
In this subsection, we analize the convergence of blowups to a minimal cone. This is crucial in the proof of the comparison principle. In particular, we prove the 
\begin{thm}
 Let $E$ minimize \eqref{minp}. Then the sets 
 $$E_\lambda := \{ x_0 + \frac{1}{\lambda}( x - x_0) \, | \, x \in E\}$$
 converge, in Hausdorff sense and up to a subsequence $\lambda_n \to 0$, to some minimizing cone $C$. In addition, for all $K$ compact of $\reg C$, there exists a neighborhood $V$ of $K$ such that $E_\lambda \cap V$ converges to $C \cap V$ in $\mathcal C^2(V)$.
\label{massaripepe}
\end{thm}
We first state a standard result about the Hausdorff convergence, which is obtained by first showing a $L^1$ convergence (see \cite{giusti84}) and using the density estimates to obtain the Hausdorff one.
\begin{prop}
Let $x_0 \in \partial E$. The sets $E_\lambda$ converge to a minimizing cone $C$ in Hausdorff distance.
\end{prop}
Now, let us inverstigate the regularity of a minimizing set which is close to $\reg C$.
\begin{prop}
Let $K$ be a compact subset of $\reg C$. Then, for every $x_0 \in K$, there exists a neighborhood $W$ of $x_0$, whose size depends only on $\Lambda = |a|$, the dimension and $K$, such that $E_\lambda \cap W$ is a $\mathcal C^2$ surface.
\end{prop}

\begin{proof}
This is proven in \cite[Th. 3]{Massari75}. Since the whole proof uses several papers (\cite{massari74,Massari75, massari752}) and does not provide information on the uniformity of the convergence, we reproduce it here. By compactness, it is enough to show that for every $x_0$, there exists a neighborhood $W$ of $x_0$ such that every $x_\lambda \in \partial E_\lambda \cap W$ belongs to $\reg E_\lambda$.
 
Since $x_0 \in \partial^\ast C$ and using the monotonicity formula for $C$ (with $a=0$), we conclude that
$$\lim_{r\to 0}r^{1-n} \left( \int_{B_r(x_0)} |D1_C| - \left \vert \int_{B_r(x_0)} D1_C \right \vert \right) = 0.$$

Choosing $r$ such that $B_r \subset K$ and
$$\int_{\partial B_r} |D1_C| = 0,$$ Lemma \ref{limTV} shows that
$$\lim_{\lambda \to \infty} \int_{\overline B_r} |D1_{E_{\lambda}}| = \int_{\overline B_r} |D1_C|.$$

In addition, using the relation (trace theorem)
$$\int_{B_r(x)} D1_{E_\lambda} = \int_{\partial B_r(x)} 1_{E_\lambda} (y) \frac{y-x}{|y-x|} dH^{n-1}(y),$$
we obtain, for almost every $r$,
$$\lim_{\lambda \to \infty} \left \vert \int_{\overline{B_r}(x_0)} D1_{E_\lambda} \right \vert = \left \vert \int_{\overline{B_r}(x_0)} D1_{C} \right \vert.$$
As a result, for every $\eps >0$, one can choose $r_0$ and $\lambda_0$ such that for all $r \ls r_0$ and $\lambda \ls \lambda_0$,
$$ r^{1-n} \left( \int_{\overline B_r(x_0)} |D1_{E_\lambda}| - \left \vert \int_{\overline B_r(x_0)} D1_{E_\lambda} \right \vert \right) \ls \varepsilon.$$

In particular, with $\varepsilon = \frac{\sigma}{2^{n-1}}$, where $\sigma$ is the constant in Theorem \ref{flatness} (corresponding to some $\alpha < 1$ that is considered fixed in what follows), we fix $\hat r \ls r_0.$ If $r_\lambda = |x_0-x_\lambda|$ for any sequence $x_\lambda \to x_0$, and choosing $\lambda_0$ such that for $\lambda \ls \lambda_0$, $r_{\lambda} \ls \frac {\hat r}2$, we have,
$$\left( \int_{\overline B_{\hat r}(x_0)} |D1_{E_\lambda}| - \left \vert \int_{\overline B_{\hat r}(x_0)} D1_{E_\lambda} \right \vert \right) \ls \sigma (\hat r-r_\lambda)^{n-1}.$$
Then, we recall that $B_{\hat r-r_\lambda}(x_\lambda) \subset B_{\hat r}(x_0)$ and notice that the integral on the left is monotone with respect to the inclusion (because for every $A \subset B_{\hat r}(x_0)$,
$$ \int_{\overline B_{\hat r}(x_0)} |D1_{E_\lambda}| - \left \vert \int_{\overline B_{\hat r}(x_0)} D1_{E_\lambda} \right \vert \gs 0 \quad ),$$
$$\int_{B_{\hat r-r_\lambda}(x_\lambda)} |D1_{E_\lambda}| - \left \vert \int_{\overline B_{\hat r-r_\lambda}(x_\lambda)} D1_{E_\lambda} \right \vert \ls \sigma (\hat r-r_\lambda)^{n-1}.$$

Let us show now that the last inequality implies that $x_\lambda \in \partial^\ast E_\lambda,$ that means there exists
\begin{equation}
\nu(x):= \lim_{r \to 0} \frac{\int_{B_r(x_\lambda)} D1_{E_\lambda}}{\int_{B_r(x_\lambda)} | D1_{E_\lambda}| } , \quad \text{with }|\nu(x)| = 1.
\end{equation}
This is \cite[Lemma 2.2]{massari752}. We introduce the notation 
$$\nu_r = \frac{\int_{B_r(x_\lambda)} D1_{E_\lambda}}{\int_{B_r(x_\lambda)} | D1_{E_\lambda}| }.$$
\begin{lem}
For every $s < r < r_0$, we have
\begin{equation}\left \vert \nu_r - \nu_s \right \vert \ls 2 \left( \frac{\int_{B_r(x_\lambda)} |D1_{E_\lambda}| - \left \vert \int_{B_r(x_\lambda)} D1_{E_\lambda} \right \vert}{\int_{B_s(x_\lambda)} | D1_{E_\lambda}| } \right)^{1/2}.
\label{dnorm}
\end{equation}
\end{lem}
\begin{proof}
Let $u$ and $v$ be smaller than 1, we have $|u-v|^2 \ls 2-2uv.$ As a result, 
$$|\nu_r - \nu_s| \ls \sqrt 2 \left( 1- \frac{\int_{B_r(x_\lambda)} D1_{E_\lambda}}{\int_{B_r(x_\lambda)} | D1_{E_\lambda}| } \cdot \frac{\int_{B_s(x_\lambda)} D1_{E_\lambda}}{\int_{B_s(x_\lambda)} | D1_{E_\lambda}| }\right)^{1/2},$$
which implies
\begin{align*} \left \vert \nu_r - \nu_s \right \vert & \ls \left( \frac{2}{\int_{B_s(x_\lambda)} | D1_{E_\lambda}| } \left( \int_{B_s(x_\lambda)} |D1_{E_\lambda}| - \frac{ \int_{B_r(x_\lambda)} D1_{E_\lambda} \cdot \int_{B_s(x_\lambda)} D1_{E_\lambda}}{\int_{B_r(x_\lambda)} |D1_{E_\lambda}|} \right) \right)^{1/2} \\
 &\ls \left( \frac{2}{\int_{B_s(x_\lambda)} | D1_{E_\lambda}| } \left( \int_{B_r(x_\lambda)} |D1_{E_\lambda}| - \frac{ \int_{B_r(x_\lambda)} D1_{E_\lambda} \cdot \int_{B_r(x_\lambda)} D1_{E_\lambda}}{\int_{B_r(x_\lambda)} |D1_{E_\lambda}|} \right) \right)^{1/2}.
\end{align*}
The last inequality is obtained using that for all $|\eta| \ls 1$ and $s \ls r$,
$$ \int_{B_s} |D1_{E_\lambda}| - \eta \cdot \int_{B_s} D1_{E_\lambda} \ls \int_{B_r} |D1_{E_\lambda}| - \eta \cdot \int_{B_r} D1_{E_\lambda}.$$
Indeed, for every $A \subset \Rn$ bounded, we have 
$$ \int_A |D1_{E_\lambda}|- \eta \cdot \int_{B_s} D1_{E_\lambda} \gs 0.$$

Finally, we get
\begin{equation*}\left \vert \nu_r - \nu_s \right \vert  \ls 2 \left( \frac{ \int_{B_r(x_\lambda)} |D 1_{E_\lambda}|}{\int_{B_s(x_\lambda)} | D1_{E_\lambda}| } \cdot \frac{\int_{B_r(x_\lambda)} |D1_{E_\lambda}| - \left \vert \int_{B_r(x_\lambda)} D1_{E_\lambda}\right \vert }{\int_{B_r(x_\lambda)} | D1_{E_\lambda}| } \right)^{1/2}
\end{equation*}
which yields \eqref{dnorm}.
\end{proof}

We will prove that $(\nu_{\alpha^k r})_{k \in \mathbb N}$ is a Cauchy sequence. Using \eqref{dnorm}, we have
$$|\nu_{\alpha^{k+m}r} - \nu_{\alpha^k r} | \ls 2 \sum_{j=0}^{m-1} \left \vert \frac{\int_{B_{\alpha^{k+j}r}(x_\lambda)} |D1_{E_\lambda}| - \int_{B_{\alpha^{k+j}r}(x_\lambda)} D1_{E_\lambda}}{\int_{B_{r\alpha^{k+j+1}}(x_\lambda)} |D1_{E_\lambda}|} \right \vert^{1/2}. $$
Thanks to Corollary \ref{corlbper}, for $r < \frac{\omega_{n-1}}{2 (n-1) \omega_n |a|}$, we have
\begin{equation} \int_{B_{\alpha^ir}(x_\lambda)} |D1_{E_\lambda}| \gs \frac{\omega_{n-1}}{2} r^{n-1} \alpha^{i(n-1)}.  \label{eqpfmas} \end{equation}
Now, Theorem \ref{flatness} implies that for $r \ls \hat r - r_\lambda$
\begin{equation}\int_{\overline B_{\alpha^i r}(x)} |D1_{E_\lambda}| - \left \vert \int_{\overline B_{\alpha^i r}(x)} D1_{E_\lambda} \right \vert \ls \alpha^{i/2} \sigma (\alpha^i r)^{n-1}.\label{eqpfmas2}\end{equation}
As a result,
\begin{align*}|\nu_{\alpha^{k+m}r} - \nu_{\alpha^k r} | &\ls 2 \sum_{j=0}^{m-1} \left ( \frac{\sigma \alpha^{(k+j)(n-\frac 12)}}{\omega_{n-1}/2 \cdot \alpha^{(k+j+1)(n-1)}}\right)^{1/2}\\
& \ls 2  \left( \frac{2 \sigma}{\omega_{n-1}}\right)^{\frac 12}  \alpha^{k/4 } \frac{1-\alpha^{m/4}}{1-\alpha^{k/4}} \\
&\ls 2  \left( \frac{2 \sigma}{\omega_{n-1}}\right)^{\frac 12} \alpha^{k/4}  \frac{1}{1-\alpha^{1/4}} ,
\end{align*}
which shows that $(\nu_{\alpha^k r})_{k \in \mathbb N}$ is a Cauchy sequence. Let $\nu(x)$ denote its limit.

Since every $|\nu_{\alpha^i r} (x) | = 1$, we have 
$$|\nu(x)| = 1.$$

Then, let us prove that in fact,
$$ \lim_{t \to 0} \nu_t(x) = \nu(x).$$
For every $t$ sufficiently small, there exists $i \in \mathbb N$ such that 
$$ r \alpha^{i+1} \ls t \ls r \alpha^i.$$
Then,
$$|\nu_t(x) - \nu(x) | \ls |\nu_t(x) - \nu_{r \alpha^i}(x) | + | \nu_{r \alpha^i}(x) - \nu(x) |.$$
Using equation \eqref{dnorm}, one can write
\begin{align*}|\nu_t(x) - \nu_{r \alpha^i}(x) | & \ls 2 \left( \frac{\int_{B_{r\alpha^i}(x_\lambda)} |D1_{E_\lambda}| - \int_{B_{r\alpha^i}(x_\lambda)} D1_{E_\lambda}}{\int_{B_t(x_\lambda)} | D1_{E_\lambda}| } \right)^{1/2} \\
& \ls 2 \left( \frac{\int_{B_{r\alpha^i}(x_\lambda)} |D1_{E_\lambda}| - \int_{B_{r\alpha^i}(x_\lambda)} D1_{E_\lambda}}{\int_{B_{r\alpha^{i+1}}(x_\lambda)} | D1_{E_\lambda}| } \right)^{1/2} \\
& \ls 2  \left( \frac{2 \sigma}{\omega_{n-1}}\right)^{\frac 12} \alpha^{i/4} \end{align*}
using Equations \eqref{eqpfmas} and \eqref{eqpfmas2}. This is exactly saying that $x_\lambda \in \partial^\ast E_\lambda$, and so in $\reg E_\lambda$ (see \cite{giusti84}, Th. 4.11).
\end{proof}
\begin{rem}
 Note that the size of $V$ depends only on the choice of $r_0$ and $\eps$, that means on the constant $\sigma$ is Lemma \ref{flatness} (so of the dimension and $|a|$) and of the convergence rate in Lemma \ref{limTV}.
\label{voisindep}
\end{rem}
We can now conclude the proof of Theorem \ref{massaripepe}. It is enough to notice that since the $E_\lambda$ have a constant mean curvature, then $\reg E_\lambda$ is in fact analytic, as well as $\reg C$. So, the local Hausdorff convergence of $E_\lambda \to C$ directly provides the $\mathcal C^2$ convergence of $E_\lambda \cap V$ to $C \cap V$.

\subsection{We can assume that \texorpdfstring{$E$}{E} and \texorpdfstring{$F$}{F} have the same tangent cone}
We are now ready to prove the strict comparison principle for constant mean curvature surfaces $E$ and $F$ (Theorem \ref{simonp}). We proceed by contradiction and assume that there exists $x_0 \in \partial E \cap \partial F$. We prove that we can assume that $E$ and $F$ have the same tangent cone at $x_0$. To do so, we use the dimension reducing argument by Bombieri and Giusti \cite{bombieri72}. Let $C_1 \subset  C_2$ the tangent cones to $E$ and $F$ at $x_0$. Then, there must exist $y \neq 0$ in $ \partial C_1 \cap \partial C_2$. Indeed, if not, we could consider a ball $B_r(x_0)$ and $C_1 \cap B_r(x_0)$ and $C_2 \cap B_r(x_0)$ would not touch near $\partial B_r(x_0)$ and would be both minimizing in $B_r(x_0)$ and contain $x_0$. We could then apply the proof of Theorem \ref{thmjer} with $E_s$ and $E_t$ replaced by $C_1$ and $C_2$ (which do not touch near the boundary of $B_r(x_0)$, which would provide a contradiction.
\\
We then blow up $C_1$ and $C_2$ at $y$ and get two tangent cones $C^1_1$ and $C^1_2$ which both contain the line $l = \mathbb R y$. Hence $D_1 = C^1_1 \cap (y + l^\perp)$ and $D_2 = C^1_2 \cap (y + l^\perp)$ are two $n-1$-dimensional minimizing cones which are either equal or distinct. If they are distinct, we can reproduce the scheme for $D_1$ and $D_2$, obtaining two $(n-2)$-dimensional minimizing cones $C^2_1$ and $C^2_2$. Since there is no singular minimizing cone with dimension smaller than 7, this iteration stops and gives two equal minimizing cones $C^k_1 = C^k_2$.

As a result, if we prove Theorem \ref{simonp} with $C_1 = C_2$, we can apply it to $C^{k-1}_1$ and $C^{k-1}_2$ which have, by definition, the same tangent cone at some point. This gives $C_1^{k-1} = C_2^{k-1}$. By (finite) induction, we will obtain $E = F$.

In what follows, we suppose that $E$ and $F$ have the same tangent cone $C$ at $x_0$. In addition, for simplicity, we take $x_0 = 0.$

\subsection{Proof}
\emph{Note that in what follows, to have the same notations as in \cite{simon87}, we use $T_1 =\partial  E$ and $T_2 = \partial F$. We also assume that $x_0 = 0$. The proof is the same as in \cite{simon87}. Nonetheless, the different blowups have no zero mean curvature anymore and we have to check that their convergence is still $\mathcal C^2$ near regular points of the limit.}
We begin by seeing that \cite[Lemma 1]{simon87} still holds with minimizers of \eqref{minp}.
\begin{lem}
 Let $E$ minimize \eqref{minp}, $x_0=0 \in \partial E$ and $\nu$ denote the unit normal to $E$. We define $\Omega_\theta$ the set of points $x \in \reg T_1$ which satisfy
\begin{enumerate}[i)]
 \item $d(x,\sing E) > \theta |x|$,
\item $$ \sup \left\{ \frac{|\nu(x) - \nu(y)|}{|x-y|} \, \middle \vert \, y \in \reg E, \,  0< |y-x| < \theta |x| \right \} < \frac{1}{\theta |x|}.$$
\end{enumerate}
Then, there exist $\rho_0(x_0,E)>0$ and $\theta_0(x_0;E)>0$ such that 
$$\forall \, 0< \rho \ls \rho_0, \ \forall \, 0< \theta \ls \theta_0, \quad \Omega_\theta \cap \partial B_\rho(x_0) \neq \emptyset.$$
\label{slem1}
\end{lem} 

\begin{proof}

 The proof is exactly the same as in \cite{simon87}. We reproduce it here and give some extra details. We proceed by contradiction. If the conclusion of the lemma were false, we could find two sequences $\rho_j \to 0$, $\theta_j \to 0$ such that
\begin{equation} \left \{ x \in \reg E \, \middle \vert \, |x| = \rho_j, \; \dist(x,\sing E) > \rho_j \theta_j, \; \sup_{\substack{y \in \reg E \\ |x-y| \ls \rho_j \theta_j }} \left [ \frac{|\nu(x)-\nu(y)|}{|x-y|} \right] < \frac{1}{\rho_j \theta_j} \right\} = \emptyset.\label{prooflem1} \end{equation}
Let $E_j = \rho_j^{-1} E.$ Thanks to Theorem \ref{massaripepe}, there exists a cone $C$, a subsequence (which we still denote by $j$) such that $E_j  \to C$ in the Hausdorff sense, and $\mathcal C^2$ sense on the neighborhoods of points in $\reg C.$ If $y \in \reg C \cap \partial B_1$ (such a point exists because $\H^{n-7}(\sing C) = 0$), there exists $\theta >0$ and a sequence $y_j \to y$ with $y_j \in B_\theta(y) \cap \partial B_1 \cap \reg E_j$ (we can take $y_j$ on the sphere again), and such that $B_\theta(y) \cap \ E_j \subset \reg E_j$ (thanks to Theorem \ref{massaripepe}). In addition, by the $\mathcal C^2$ convergence (and eventually reducing $\theta$ again), one can have
$$\forall x,z \in B_\theta(y) \cap \reg E_j, \quad \frac{|\nu(x) - \nu(z)|}{|x-z|} \ls \frac 1\theta.$$
Going back to $E$, we have 
\begin{equation}\forall x,z \in B_{\rho_j\theta}(\rho_j y) \cap \reg E, \quad \frac{|\nu(x) - \nu(z)|}{|x-z|} \ls \frac 1{\rho_j\theta}.\label{eqrege} \end{equation}
Finally, notice that $\rho_j y_j \in \partial B_{\rho_j} \cap  \reg E$. In addition, $$\dist(\rho_j y_j,\sing E) = \rho_j \dist(y_j,\sing E_j) \gs \rho_j \theta$$ and, using \eqref{eqrege} with $z = \rho_j y_j \in B_{\rho_j \theta}(\rho_j y)$, this contradicts \eqref{prooflem1} for $j$ large enough.

\end{proof}

Let $\rho_0, \theta_0$ and $\Omega_\theta \subset \reg T_1$ as in Lemma \ref{slem1} and define, for all $x \in T_1$, $$h(x) = \dist(x,\spt T_2).$$ Since $T_1$ and $T_2$ have the same tangent cones at $x_0$, one has, for every $\theta \ls \theta_0$,
\begin{equation}\lim_{r \to 0} r^{-1} \sup_{|x| = r, x \in \Omega_\theta} h(x) = 0.\label{mmc}\end{equation}
Indeed, we have in fact
$$ \frac{1}{r} \sup_{|x| = r, x \in T_1} d(x,C) = \sup_{|y|=1, y \in r^{-1} T_1} d( y, C) \to 0$$ because of Hausdorff convergence of $r^{-1} T_1$ to $C$.
As the same holds for $x \in T_2$, that gives
$$ \frac{1}{r} \sup_{|x| = r, x \in T_2} d(x,C) = \sup_{|y|=1, y \in r^{-1} T_2} d( y, C) \to 0$$ 
which implies \eqref{mmc}.\\
We select $\rho_j \to 0$ such that for all $\rho \ls \rho_j$,
$$\rho_j^{-1} \sup_{\substack{x\in \Omega_{\theta_0} \\ |x| = \rho_j}} h(x) \gs \frac 12 \rho^{-1} \sup_{\substack{x\in \Omega_{\theta_0} \\ |x| = \rho}}h(x)$$
we have in particular for $\theta < 1$,
\begin{equation}\sup_{\substack{x\in \Omega_{\theta_0} \\ |x| = \theta \rho_j}} h(x) \ls 2 \theta \sup_{\substack{x\in \Omega_{\theta_0} \\ |x| = \rho_j}} h(x).\label{sim2}\end{equation}

Let $\rho_j \to 0$ and $T_l^{(j)} = \rho_j^{-1} T_l.$ We want to show that $T_l^{(j)}$ are normal graphs over points of $\reg C$. 
\begin{lem}
For every $ l \in {1,2}$, there exist a sequence of $\mathcal C^2$ functions $h_l^{(j)}$ which is defined in a connected domain $U_j$ such that for some $\theta_j \to 0$,
\begin{equation} \left \{ x \in \reg C\, \middle | \, \operatorname{dist}(x, \sing C) > \theta_j |x|, \; \theta_j < |x| < \theta_j^{-1} \right \} \subset U_j \label{sim3}\end{equation}
and such that 
\begin{equation} \lim_{j \to +\infty} |h_l^{(j)}|_{\mathcal C^2}^\ast = 0, \quad \text{with} \quad \vert f \vert_{\mathcal C^2}^\ast := \sup \frac{|f(x)|}{|x|} + |\nabla f(x)|  + |x| |\nabla^2 f(x)|.\label{sim3b}\end{equation}
and that for every $\theta \in (0,1)$ and every $j \gs j(\theta)$, we also have, for $l \in \{1,2\}$,
\begin{equation}
 \left \{ x \in \reg (\rho_j^{-1} T_l) \, \middle | \, \operatorname{dist}(x, \sing C) > \theta |x|, \; \theta < |x| < \theta^{-1} \right \} \subset G_l^{(j)} \subset \reg (\rho_j^{-1} T_l)
\label{sim4}
\end{equation}
where $G_l^{(j)}$ is the graph of $h_l^{(j)}$ (more precisely, $G_l^{(j)} = H_l^{(j)} (U_j)$ where $H_l^{(j)}(x)  = x + h_l^{(j)} \nu(x)$ and $\nu(x)$ is the normal of $\reg(C)$ at $x$).
We also ask that
\begin{equation}
\rho_j^{-1}(\Omega_{2\theta}) \cap \left \{ x \, \middle | \, \theta < |x| < \theta^{-1} \right \} \subset \left \{ x \in \reg (\rho_j^{-1} T_1) \, \middle | \, \dist(x,\sing C) > \theta |x| \right\}.
 \label{sim5}
\end{equation}

\end{lem}

\begin{proof}
Let $T_l^{(j)} := \rho_j^{-1} T_l$. We construct $\theta_j$ as follows. Let $\theta_1$ be any real in $(0,1)$ and for $l \in \{1,2\}$,
we consider
$$K_1 := \{x \in \reg C \, \vert \, \dist(x,\sing C) \gs \theta_1 |x|, \; \theta_1 \ls |x| \ls \theta_1^{-1} \}.$$
It is a compact subset of $\reg C$. Thanks to Theorem \ref{massaripepe}, there exists $h_1$ such that if $y \in T_l^{(j)}$ satisfies $|y-x| < h_1$ for some $x \in K_1$, then $y \in \reg T_l^{(j)}$.

Using the Hausdorff convergence of $T_l^{(j)}$ to $C$ on the compact set
$$L_1 = \{x \in \Rn \, \vert \, , \dist(x,\sing C) \gs \theta_1 |x|, \; \theta_1 \ls |x| \ls \theta_1^{-1}  \},$$
there exists $j_2$ such that for every $j\gs j_2$ and $y \in L_1 \cap T_l^{(j)}$, there exists $x \in K_1$ with $|x-y| \ls h_0/2.$ That implies that 
$$ L_1 \cap T_l^{(j)} \subset \reg T_l^{(j)}.$$
We can increase $j_2$ again such that $L_1 \cap T_l^{(j)}$ is in fact a graph of $h_l^{(j)}$ over $K_1$ with 
$$\Vert h_l^{(j)} \Vert_{\mathcal C^2} \ls \frac{1}{j}.$$
This is possible since the $L^\infty$ convergence of the $h_l^{(j)}$ is provided by the Hausdorff convergence of $T_l^{(j)}$ to $C$ and the $\mathcal C^2$ is obtained using the analyticity of $\reg T_l^{(j)}$ as well as $\reg C$.
We let $\theta_{j_2} = \frac{\theta_1}{2}$ and for every $j \in [1,j_2-1]$, $\theta_j = \theta_1.$ To define $j_3$, we use the same scheme with $\theta_{j_2}$ in place of $\theta_1$: that enables to define $\theta_k$ for $k \ls j_3$. Then, $\theta_j \to 0$.

We proved \eqref{sim3b} and \eqref{sim4} by construction.

We now prove \eqref{sim5}. If it does not hold, then there exists $\theta$ and $j_k \to \infty$ such that there exists 
$$x_k \in \rho_{j_k}^{-1}(\Omega_{2\theta}) \cap \left \{ x \, \middle | \, \theta < |x| < \theta^{-1} \right \}$$ 
and
$$\dist(x_k,\sing C) \ls \theta |x_k|.$$
The last equation means that there is $z_k \in \sing C$ such that $|x_k - z_k| \ls \theta |x_k|.$ One can assume that $z_k \to \overline z \in \sing C$ using the local compactness of $\sing C$. Finally, $|x_k - \overline z| \ls \theta |x_k| + \eps_k$ with $\eps_k \to 0$.

The point $\overline z$ is singular, which implies in particular that $C$ cannot be a graph around it. As a result, we have a unit vector $\nu$ and two sequences $z^i,\tilde z^i \in \reg C$ which converge to $\overline z$ and whose normals $\nu(z^i)$ and $\nu(\tilde z^i)$ converge respectively to $\nu$ and $-\nu$. Since $T^j$ converge $\mathcal C^2$ to $C$ in the neighborhood of $\reg C$, there exist (using a diagonal argument) $\alpha_k ,\tilde \alpha_{k} \in \reg T_1^{(j_k)}$ such that 
$$|\overline z - \alpha_{k}| \ls \frac{\theta^2}{10}$$ and the normals $\nu(\alpha_{k})$ and $ \nu(\tilde \alpha_{k})$ to $\reg T_1^{j_k}$ satisfy, for $k$ large enough,
\begin{equation}|\nu(\alpha_{k}) - \nu(\tilde \alpha_{k}) | \gs \frac{3}{2}.\label{eqpasreg}\end{equation}
On the other hand, since $x_k \in \rho_{j_k}^{-1}\Omega_{2\theta}$, we have 
$$ \sup \left\{ \frac{|\nu(x_k) - \nu(y)|}{|x_k-y|} \, \middle \vert \, y \in \reg T_1^{j_k}, \,  0< |y-x_k| < 2\theta |x_k| \right \} < \frac{1}{2\theta |x_k|}.$$
Noting that we can choose $y = \alpha_{k}$ and $y= \tilde \alpha_{k}$ in the last identity, it provides
$$\frac{|\nu(x_k) - \nu(\alpha_{k})|}{|x_k-\alpha_{k}|} \ls \frac{1}{2\theta |x_k|}$$
which implies, since $|x_k - \alpha_{k}| \ls 11 \frac{\theta |x_k|}{10}$,
$$|\nu(x_k) - \nu(\alpha_{k})| \ls \frac{11}{20}.$$
The same holds for $\tilde \alpha_{k}$. Summing, we get
$$|\nu(\alpha_{k}) - \nu(\tilde \alpha_{k})| \ls \frac{11}{10}.$$
This contradicts \eqref{eqpasreg}, proving \eqref{sim5}.
\end{proof}

Using this lemma, we have maps $$ p_j : \rho_j^{-1}\Omega_{2\theta}\cap \left \{ x \, \middle | \, \theta < |x| < \theta^{-1} \right \} \to U_j $$
with
$$H_1^{(j)} (p_j(x)) = p_j(x) + h_1^{(j)}(p_j(x)) \nu(p_j(x)) = x$$
and $ \forall \left \{ x \, \middle | \, \theta < |x| < \theta^{-1} \right \}$ and $j$ sufficiently large, 
\begin{equation}\quad \frac{1}{2} u_j(p_j(x))  \ls \rho_j^{-1} \dist(\rho_j x,T_2) \ls 2 u_j(p_j(x)) \label{distgraph} \end{equation}
where $u_j = h_1^{(j)} - h_2^{(j)}$. The last inequality is provided by the convergence of $\nu_l^{(j)}(x_j)$ to $\nu(x)$ for $x_j \to x$ (and obvious notation). We notice that since $\reg T_1 \cap \reg T_2 = \emptyset$ (Thanks to the strong maximum principle for regular surfaces), one can assume that $u_j > 0.$ Equation \eqref{sim2}, after dilation with a factor $\rho_j^{(-1)}$, gives
$$ \sup_{\substack{x \in \rho_j^{-1} \Omega_{\theta_0} \\ |x| = \theta}} \rho_j^{-1}h(\rho_j x) \ls 2 \theta  \sup_{\substack{x \in \rho_j^{-1} \Omega_{\theta_0} \\ |x| = 1}} \rho_j^{-1}h(\rho_j x)$$
Using then \eqref{distgraph}, we obtain
\begin{equation} \label{contrad2} \sup_{\substack{x \in \rho_j^{-1} \Omega_{\theta_0} \\ |x| = \theta}} u_j(p_j(x)) \ls 4 \theta \sup_{\substack{x \in \rho_j^{-1} \Omega_{\theta_0} \\ |x| = 1}} u_j(p_j(x)).\end{equation}

Since $\reg T_1$ and $\reg T_2$ are two constant mean curvature submanifolds, we can prove the
\begin{lem}
 The difference $u_j := h_1^{(j)} - h_2 ^{(j)}$ satisfies an equation of the form
\begin{equation}\Delta_C u_j + |A_C|^2 u_j = \div(a_j\cdot \nabla u_j) + b_j \cdot \nabla u_j + c_j u_j\label{eqjacobi}\end{equation}
where $\Delta_C$ is the Laplace-Beltrami operator on $C$, $A_C$ the second fundamental form of $C$ and $a_j,b_j,c_j$ three functions converging uniformly to zero on compact subsets of $\reg C$.
\label{jacobi}
\end{lem}
\begin{proof}
Let $f$ be a function on $\reg C$ and consider $M$ the normal graph of $f$ over $\reg C$ (we note only $C$ in the rest of the proof). 
A parametrization of $M$ is (locally) $F : \Omega \to \R^n$ with
$$F(x) = C(x) + f(C(x)) \nu(C(x)) $$
where $C(x)$ is a local parametrization of $C$. More precisely, the metric on $C$ is written
$$g_{ij} = \scal{\partial_i C}{\partial_j C}.$$
As a result, a tangent vector is written (the $h_{\alpha,\beta}$ are the coefficients of $A_C$)

$$\tau_i = \partial_i F = \partial_i (C + f \nu) = \partial_i C + \partial_i f \nu + f \partial_i \nu = \partial_i C + \partial_i f \nu + f h_{il}g^{lm} \partial_m C.$$

Thus the metric on $M$ is
\begin{align}
 \tilde g_{ij} & = \scal{\partial_i C + \partial_i f \nu + f h_{il}g^{lm} \partial_m C}{ \partial_j C + \partial_j f \nu + f h_{js}g^{st} \partial_t C} \\
&= g_{ij} + \partial_i f \partial_j f + f h_{il} g^{lm} g_{mj} + f h_{js} g^{st} g_{it} + f^2 h_{il}g^{lm}h_{js}g^{st} g_{mt} \\
&=g_{ij} + f \left( h_{il} g^{lj} + h_{jl} g^{li} \right) + \partial_i f \partial_j f + f^2 g^{lm} h_{il} h_{jm}.
\end{align}
Note that this metric does not contain any derivatives of order two for $f$. Using normal coordinates on $C$, it can be rewritten as
$$\tilde g_{ij} = \delta_{ij}(1 + 2f \lambda_i + f^2 \lambda_i^2) + \partial_i f \partial_j f.$$

The normal to $M$ can be computed in the basis $(\partial_i C, \nu)$ as 
$$\tilde \nu = \alpha \nu + \sum_i \beta_i \partial_i C.$$
The coefficients $\alpha$ and $\beta_i$ satisfy

\begin{equation}
\begin{aligned}
 0&=\scal{\tilde \nu}{\partial_i F} = \scal{\tilde \nu}{\partial_i C + \partial_i f \nu + f h_{il}g^{lm} \partial_m C} \\
&=\beta_j g_{ij} + \alpha \partial_i f + f h_{il} g^{lm} \beta_j g_{jm} \\
&= \beta_j g_{ij} + \alpha \partial_i f + f h_{il} \beta_l.
\end{aligned}
\label{eqalphabeta1}
\end{equation}

and 
\begin{equation}\scal{\tilde \nu}{\tilde \nu} = \alpha^2 + \beta_i \beta_j g_{ij} = 1.
\label{eqalphabeta2}
\end{equation}

So, the coefficients $\alpha$ and $\beta$ depends only on order zero and one derivatives of $f$.

One also have $\tilde h_{ij} = - \scal{F_{ij}}{\tilde \nu}.$ Let us compute 
\begin{align*}
 \partial_j(F_i)&= \partial_j \left( \partial_i C + \partial_i f \nu + f h_{il}g^{lm} \partial_m C \right) \\
&= \partial_{ij} C + \partial_{ij} f \nu + \partial_i f \partial_j \nu + \partial_j f h_{il}g^{lm} \partial_m C + f \partial_j(h_{jl}) g^{lm} \partial_m C \\ 
& \quad +f h_{il} \partial_j(g^{lm}) \partial_m C + f h_{il}g^{lm} \partial_{jm} C.
\end{align*}
Hence 
\begin{align*}
  \scal{F_{ij}}{\tilde \nu} &= \scal{\partial_{ij} C}{\alpha \nu + \beta_k \partial_k C} + \alpha \partial_{ij} f+ \partial_i f \scal{\partial_j \nu}{\beta_k \partial_k C} \\
  &  + \partial_j f h_{il} g^{lm} \beta_k \scal{\partial_k C}{\partial_m C} + f \partial_j(h_{il}) g^{lm} \beta_k \scal{\partial_k C}{\partial_m C} \\
  &+ f h_{il} \partial_j (g^{lm}) \beta_k \scal{\partial_k C}{\partial_m C} + f h_{il}g^{lm} \scal{\partial_{jm} C}{\alpha n + \beta_k C_k}.
\end{align*}

In normal coordinates on $C$ (the second fundamental form is written $h_{ij} = \lambda_i \delta_{ij}$), that can be rewritten as
\begin{align*}
 \tilde h_{ij} &=\alpha h_{ij} - \alpha \partial_{ij} f -\partial_i f \lambda_j \beta_j - \partial_j f \lambda_i \beta_i - f \partial_j(h_{il})\beta_l + \alpha f h_{ij}.
\end{align*}
To compute the mean curvature, we need the inverse of the metric. We compute using normal coordinates in $C$.

\begin{align*}
 \tilde g^{ij} &= \delta_ij \left(1 - 2 f\lambda_i -f^2 \lambda_i^2 \right) - \partial_i f \partial_j f + 4 f^2 \delta_{ij} \lambda_i^2 +o(f^2) \\
& = \delta_{ij} (1-2f \lambda_i + 3 f^2 \lambda_i^2) - \partial_i f \partial_j f + o(f^2).
\end{align*}
Note that no term in this metric (even in $o(f^2)$) involves second derivative of $f$. We have to estimate $\alpha$ and $\beta_i$. In normal coordinates, we have, using \eqref{eqalphabeta1} and \eqref{eqalphabeta2},
$$\beta_i +\alpha \partial_i f + f \lambda_i \beta_i =0,$$
which yields
$$\beta_i = -\frac{\alpha \partial_i f}{1+f\lambda_i} = - \alpha \partial_i f (1-f\lambda_i+ o(f)) = -\alpha \partial_i f + \alpha f \partial_i f \lambda_i + o(f^2).$$
On the other hand, $\alpha^2 + \sum \beta_i^2 = 1$, which means
$$\alpha^2 + \sum_i (-\alpha \partial_i f + f\partial_i f \lambda_i + o(f^2))^2 = 1,$$
or
$$\alpha^2 \left( 1+ \sum_i (\partial_i f)^2 (1+ f\lambda_i o(f^2))^2 \right) = \alpha^2 (1+|\nabla f|^2) + o(f^2) =1.$$
Finally,
$$\alpha = \sqrt{\frac{1}{1+ |\nabla f|^2}} = 1-\frac 12 |\nabla f|^2 + o(f^2)$$
and
$$\beta_i = - \partial_i f +f \partial_i f \lambda_i +o(f^2)$$
where there is no second derivative of $f$ in $o(f^2)$.

The mean curvature can now be computed using normal coordinates on $C$ (once again, no second derivative in $o(f^2)$).
\begin{align*}
\tilde H &= \tilde g^{ij} \tilde h_{ij} = (\delta_{ij} (1-2f \lambda_i + 3 f^2 \lambda_i^2) - \partial_i f \partial_j f ) \cdot \\
& \cdot (\alpha h_{ij} - \alpha \partial_{ij} f -\partial_i f \lambda_j \beta_j - \partial_j f \lambda_i \beta_i - f \partial_j(h_{il})\beta_l + \alpha f h_{ij}) \\
& =  \alpha \lambda_i (1-2f \lambda_i-3f^2 \lambda_i^2) + \alpha( -\partial_{ii} f +2f \lambda_i \partial_{ii} f) - 2\partial_i f \lambda_i \beta_i -f \partial_i(h_{il}) \beta_l \\
& \quad + \alpha f \lambda_i  - 2f^2 \alpha \lambda_i^2 - \alpha \lambda_i \partial_i f^2 +o(f^2) \\
 & =  \lambda_i (1-2f \lambda_i-3f^2 \lambda_i^2) - \frac 12 |\nabla f|^2 \lambda_i -\partial_{ii} f +2f \lambda_i \partial_{ii} f \\
& \quad   + 2 \lambda_i (\partial_i f)^2 + f \partial_i(h_{il}) \partial_l f + f \lambda_i -2f^2 \lambda_i^2 - \lambda_i (\partial_i f)^2 + o(f^2) \\
&= \underbrace{(1+f)H - \frac 12 |\nabla f|^2 H}_{=0} - \Delta f -2f |A|^2  + \div(a \nabla f) + b \cdot \nabla f +c f 
\end{align*}
with
$$a_{ij} = 2f h_{ij} \quad \text{and} \quad b_i =- \lambda_i \partial_i f - f \partial_k(h_{ki})+o(f) \quad \text{and} \quad c = -2f |A|^2 + o(f).$$
So, $h_l^{(j)}$ both satisfy (we denote by $H_j$ the (constant) mean curvature of $T_l^{(j)}$)
$$ H_j +  \Delta h_l^{(j)} +2h_l^{(j)} |A|^2 =    \div(a \nabla f) + b \cdot \nabla f +c f.$$

Substracting the two equations (and denoting by $B$ the quantity $2A$) and noting that since the two terms $o(h_1^{(j)})$ and $o(h_2^{(j)})$ are regular and obtained by the same procedure, one has $o(h_1^{(j)}) - o(h_2^{(j)}) = o(u_j)$, we get
\begin{align*}  \Delta u_j + 2 u_j |A|^2 & = \div\left( h_1^{(j)} B \nabla (h_1^{(j)}) - h_2^{(j)}  B \nabla (h_2^{(j)}) \right) \\ & - \sum_i \lambda_i \left[ (\partial_i h_1^{(j)})^2 - (\partial_i h_2^{(j)})^2 \right] -2 u_j(h_1^{(j)} + h_2^{(j)}) |A|^2 + o(u_j) 
\\
& = \div\left(  (h_1^{(j)} + h_2^{(j)}) B \nabla u_j \right) + \div\left( h_1^{(j)} B \nabla h_2^{(j)} -h_2^{(j)} B \nabla h_1^{(j)} \right) \\
& - \sum_{i} \lambda_i \partial_i(h_1^{(j)} - h_2^{(j)}) \partial_i(h_1^{(j)} + h_2^{(j)}) -2 u_j(h_1^{(j)} + h_2^{(j)}) |A|^2 + o(u_j).
\end{align*}
Then, we write
$$ \div\left( h_1^{(j)} B \nabla h_2^{(j)} -h_2^{(j)} B \nabla h_1^{(j)} \right) = h_1^{(j)} (\div (B \nabla u_j)) + (h_1^{(j)} - h_2^{(j)}) \div(B\nabla h_1^j)$$
and 
$$h_1^{(j)} (\div (B \nabla u_j)) = \div(h_1^{(j)} B \nabla u_j) - \nabla h_1^{(j)}\cdot  B \nabla u_j$$
to get
\begin{align*}\Delta u_j + 2 u_j |A|^2 & = \div\left( ( 2h_1^{(j)} + h_2^{(j)}) B \nabla u_j \right)  \\
& + \left( B \nabla h_1^{(j)} - A\nabla (h_1^{(j)} + h_2^{(j)}) \right) \cdot \nabla u_j \\
& +\left(  -2 u_j(h_1^{(j)} + h_2^{(j)}) |A|^2 + \div(B\nabla h_1^{(j)}) \right) u_j + o(u_j).
\end{align*}
Then, it remains to see that with
$$a_j:= (2h_1^{(j)} + h_2^{(j)}) B $$
$$ b_j :=  B \nabla h_1^{(j)} - A\nabla (h_1^{(j)} + h_2^{(j)})$$
and 
$$c_j := -2 u_j(h_1^{(j)} + h_2^{(j)}) |A|^2 + \div(B\nabla h_1^{(j)}) +\varepsilon_j$$
where $o(u_j) = \varepsilon_j u_j$, we have $a_j,b_j,c_j \to 0$ on compact subsets of $\reg C$ and satisfy \eqref{eqjacobi}.
\end{proof}
The rest of the proof is similar to \cite{simon87}. Nonetheless, we reproduce it for convenience (and give extra details).

Since $u_j >0$, one can use Harnack inequality in \eqref{eqjacobi} on a compact $K \subset \reg C$. It yields 
\begin{equation} \label{sim8} \sup_{K} u_j \ls c_K \inf_K u_j. \end{equation}
Then, Schauder theory (\cite[Th. 8.32]{gilbarg01}) implies that for $j$ large enough,
$$|u_j|_{\mathcal C^{1,\alpha}(K)} \ls c_K \inf_K u_j.$$
Now, let us fix $y_0 \in \reg C$. Then, the sequence $\alpha_j :=(u_j(y_0)^{-1}) u_j$ converges, up to a subsequence, in $\mathcal C^1_{loc}(\reg C)$ to some function $u$. Since $\alpha_j(y_0) = 1$ for all $j$, $|\alpha_j|_{\mathcal C^{1,\alpha}(K)}$ is bounded away from zero, and so is $\inf_K u_j.$ As a result, $u >0$ on $\reg C$ (and $u(y_0)=1$). On the other hand, $u$ is a solution of
$$\Delta_C u +|A_C|^2 u = 0.$$
In particular, $\Delta_C u \ls 0$ on $\reg C$.

The last part of the proof consists in applying Bombieri and Giusti Harnack inequality \cite[Th. 6]{bombieri72} for functions on a minimal cone to $u$ on $\reg C$. 
\begin{lem}
 There exists a sequence $\varphi_j \in \mathcal C^\infty_c(\reg C)$ such that
\begin{itemize}
\item For every $x \in \Omega$, $0 \ls \varphi_j(x) \ls 1$
 \item For every $x \in \reg C$ such that $\frac{1}{j} \ls |x| \ls j$ and $\dist(x,\sing C) > \frac{1}{j}$, we have $\varphi_j(x) = 1$,
\item For a fixed $R >0$, one has
\begin{equation}\int_{\reg C \cap B_R(0)} |\nabla \varphi_j|^2 \to 0.\label{l2phii} \end{equation}
\end{itemize}

\end{lem}
\begin{proof}
First, note that $\mathcal H^{n-2}(\sing C) = 0$, so, for all $\varepsilon >0$, we can cover $\sing C$ by $N_j$ balls $B_i := B_{\rho_i}(x_i)$, of radius $\rho_i$ such that
$$\rho_i \ls \frac{1}{2j} \quad \text{and} \quad \sum_i \rho_i^{n-2} \ls \varepsilon.$$
We take $\eps = \frac 1 j$ in what follows. \\
For every $i$, we introduce a smooth function $\psi_i$ such that $\psi_i(x) = 1$ on $B_i$ and $\psi_i= 0$ on $\Omega \setminus B_{2\rho_i}(x_i).$
Then,
$$\int_{\Omega} |\nabla \psi_i|^2 \ls \left(\frac{\rho_i}{2}\right)^{-2} (4\rho_i^n - \rho_i^n) \ls \rho_i^{n-2}.$$
We introduce $\psi^j := 1 - \max(\psi_i).$ Then, as soon as $\dist(x,\sing C) > \frac{1}{j}$, $\dist(x,B_i) > \frac{1}{2j} > \rho_j$ so $\psi_i(x) = 0$ and then $\psi^j(x) = 1.$ \\
Let us define the sets $A_0 = \emptyset$ and
$$\forall 1 \ls i \ls N_j, \quad A_i := \{ \psi^j = 1- \psi_i \} \setminus \bigcup_{k < i} A_k.$$
One can compute
$$\int_{\Omega} |\nabla \psi_j|^2 = \sum_i \int_{A_i} |\nabla \psi_i|^2 \ls  \sum_i \rho_i^{n-2} \ls \frac 1j.$$
Finally, we set
$$\varphi_j = \chi_j \circ \psi^j$$ where $\chi_j$ is a cut off function such that $ \chi_j = 1$ on $B_j(0)$ and $0$ on $\Omega \setminus B_{j+1}$. This way, $|\nabla \chi_j| \ls 1$. As a result, 
$$\forall x \in \Omega, \quad |\nabla \varphi_j(x)| \ls |\nabla \psi^j (x)|$$ and $\varphi_j$ fulfills the requirement of the lemma.
\end{proof}
 
Now, let $Q >0$ and $u_Q = \min(u,Q).$ Since $\Delta_C u \ls 0$ on $\reg C$, one has, for every $\zeta \gs 0$ Lipschitz compactly supported on $\reg C$,
$$\int_{\reg C} \nabla u_Q \cdot \nabla \zeta \gs 0.$$ Let $\psi \in \mathcal C^\infty_c(\R^n).$ With $\zeta = \varphi_j^2 \psi^2 u_Q^{-1}$, we have
$$\int_{\reg C} \nabla u_Q \cdot \left(2 u_Q^{-1} \psi \nabla \varphi_j +2 u_Q^{-1} \varphi_j^2  \nabla \psi^2 - \frac{1}{u_Q} \varphi_j^2 \psi^2 \nabla u_Q \right) \gs 0.$$
Using the regularity of $u$, \eqref{l2phii} and $\varphi_j \to 1$ uniformly on compact sets of $\reg C$, we get that for every $R>0$,
$$\int_{B_R(0) \cap \reg C} |\nabla u_Q|^2 < + \infty.$$
On the other hand, with $\zeta = \psi \varphi_j$ and assuming $\psi \gs 0$, and letting $j \to \infty$ we obtain
$$\int_{\reg C} \nabla u_Q \cdot \nabla \psi \gs 0.$$

Thanks to the two last inequalities, one can now apply \cite[Th. 6]{bombieri72} with $p = 1$, which tells that
$$\int_{\reg C \cap B_2(0)} u_Q \ls c \inf_{\reg C \cap B_2(0)} u_Q.$$
With $Q \to \infty$, we obtain
$$\int_{\reg C \cap B_2(0)} u \ls c \inf_{\reg C \cap B_2(0)} u >0.$$

Coming back to the functions $u_j$, on every (non empty) compact $L \subset \reg C \cap B_2(0)$, we have
$$\inf_L u \gs \inf_{\reg C \cap B_2(0)} u := \delta >0.$$
As $\inf_L u_j \to \inf_L u$, one has, for $j$ larger than some $j_1$,
$$\inf_L u_j \gs \frac{\delta}{2}.$$
On the other hand, $u_j(y_0) \to u(y_0) = 1.$ So, there exists $j_2$ such that $\forall j \gs j_2$, $u_j(y_0) \gs \frac 12.$
Thus, there exists $j_3 = \max(j_1,j_2)$ such that for all $j \gs j_3$,
$$\inf_L u_j \gs \frac{\delta}{4} u_j(y_0).$$
Remembering \eqref{sim8}, one deduce that for every $K \subset \reg C \cap B_2(0)$ compact (non empty), one has, for $j$ sufficiently large (depending on $K$ and $L$),
\begin{equation}\inf_L u_j \gs c_K \sup_K u_j.\label{contrad1}\end{equation}
Taking $K = p_j(\rho_j^{-1} \Omega_{\theta_0} \cap \partial B_1)$ and $L = p_j(\rho_j^{-1} \Omega_{\theta_0} \cap \partial B_\theta),$ we see that for small $\theta$, \eqref{contrad1} and \eqref{contrad2} cannot happen together. This is a contradiction.

\subsection*{Acknowledgment}
I am very grateful to Antonin Chambolle for introducing me to these problems and for many fruitful discussions. I thank Giovanni Bellettini for his interest in my work and for pointing out Allard's result on the absence of monotonicity formula in the anisotropic framework.

\end{document}